\documentclass[11pt,oneside,english]{amsart}
\usepackage[T1]{fontenc}
\usepackage[latin9]{inputenc}
\usepackage{geometry}
\usepackage{amsthm}
\usepackage{amssymb}
\usepackage{esint}
\usepackage{marginnote}

 \RequirePackage[colorlinks]{hyperref}
\makeatletter
\numberwithin{equation}{section}
\numberwithin{figure}{section}
\theoremstyle{plain}
\newtheorem{thm}{\protect\theoremname}[section]
\newtheorem{cor}[thm]{\protect\corollaryname}
\theoremstyle{definition}
\newtheorem{defn}[thm]{\protect\definitionname}
\theoremstyle{plain}
\newtheorem{lem}[thm]{\protect\lemmaname}

\newtheorem{proposition}[thm]{\sc Proposition}

\newcommand{\be}{\begin{equation}}
\newcommand{\ee}{\end{equation}}
\newcommand{\fl}[1]{\lfloor{#1}\rfloor}

\newcommand{\I}{{\rm i}}

\def\eps{\varepsilon}
\def\Z{\mathbb{Z}} \def\N{\mathbb{N}}   \def\R{\mathbb{R}}

\usepackage{tikz}
\usetikzlibrary{arrows}
\newcommand{\midarrow}{\tikz \draw[-triangle 90] (0,0) -- +(.1,0);}

\makeatother

\usepackage{babel}
  \providecommand{\definitionname}{Definition}
  \providecommand{\lemmaname}{Lemma}
\providecommand{\theoremname}{Theorem}
\providecommand{\corollaryname}{Corollary}
\begin{document}

%\title{Strict-weak polymers}
\title{The strict-weak lattice polymer}

%\author{Ivan Corwin$^1$, Timo Seppäläinen$^2$, Hao Shen$^3$}

\author[I. Corwin]{Ivan Corwin}
\address{I. Corwin, Columbia University,
Department of Mathematics,
2990 Broadway,
New York, NY 10027, USA,
and Clay Mathematics Institute, 10 Memorial Blvd. Suite 902, Providence, RI 02903, USA,
and Institute Henri Poincare, 11 Rue Pierre et Marie Curie, 75005 Paris, France,
and Massachusetts Institute of Technology,
Department of Mathematics,
77 Massachusetts Avenue, Cambridge, MA 02139-4307, USA}
\email{ivan.corwin@gmail.com}

\author[T.Seppäläinen]{Timo Seppäläinen}
\address{T.Seppäläinen,
University of Wisconsin-Madison,
Department of Mathematics,
425 Van Vleck Hall, Madison, WI 53706-1388, USA}
\email{seppalai@math.wisc.edu}

\author[H.Shen]{Hao Shen}
\address{H.Shen,
University of Warwick,
Mathematics Department,
Coventry, CV4 7AL, UK}
\email{pkushenhao@gmail.com}

\begin{abstract}
We introduce the strict-weak polymer model,
and show the KPZ universality of the free energy fluctuation of this model for a certain range of parameters. Our proof relies on the observation that the discrete time geometric $q$-TASEP model, studied earlier by A.~Borodin and I.~Corwin, scales to this polymer model in the limit $q\to 1$. This allows us to exploit the exact results for geometric $q$-TASEP to derive a Fredholm determinant formula for the strict-weak polymer, and in turn perform rigorous asymptotic analysis to show KPZ scaling and GUE Tracy-Widom limit for the free energy fluctuations.
We also derive moments formulae for the polymer partition function directly by Bethe ansatz, and identify the limit of the free energy using a stationary version
of the polymer model.
\end{abstract}

\maketitle

\section{Introduction and results}

In this paper we introduce the  exactly solvable strict-weak polymer model on the two-dimensional square lattice, and investigate some of its  features.  This brings the number of known    exactly solvable directed lattice polymer models to two. (Subsequent to this paper, a generalization of this model called the Beta Polymer was discovered and analyzed in \cite{Barraquand2015}).
The strict-weak  model introduced here differs from the earlier studied log-gamma polymer \cite{MR3116323, COSZ, sepp-12-aop} in the definition of the admissible polymer paths.   The strict-weak model uses gamma-distributed weights on the edges (or vertices, depending on the formulation chosen) while the log-gamma polymer  uses inverse gamma weights.

 We show
that the strict-weak model belongs to the Kardar-Parisi-Zhang (KPZ) universality class by deriving the Tracy-Widom GUE limit distribution for the fluctuations of the free energy.  This result is based on the fact that, under an appropriate scaling of parameters and scaling and centering of the variables,  the geometric $q$-TASEP particle system converges to the strict-weak polymer. This  allows us to write a Fredholm determinant formula for the Laplace transform of the  strict-weak polymer partition function.

 We also derive an integral formula for the moments of the partition function via the rigorous replica method.   Finally, we show that this model has a stationary version where the ratios of nearest-neighbor pairs of partition functions are gamma-distributed. We use the stationary model to give an alternative derivation of the explicit  limiting free energy density, which also arises in the proof of the free energy fluctuations.

The Tracy-Widom limit of the strict-weak polymer model was proved independently and concurrently by O'Connell and Ortmann   \cite{OConOrth}.  They  derived the Fredholm determinant formula (our Theorem \ref{thm:Fredholm-SWpolymer}) in a different way that   complements our work.  They use previous work of \cite{COSZ} on the geometric RSK correspondence to relate the strict-weak polymer to a particular Whittaker process. Then, using an identity from \cite{COSZ} and a variant of an argument from \cite{MR3116323}, they arrive at the result of Theorem \ref{thm:Fredholm-SWpolymer}.

 \medskip

We turn to the definition of the model and the main results.
Our convention is to define the model
on a two dimensional $t-n$ lattice. The variable $t$ represents
  discrete time, with the axis pointing to the right. The variable
$n$ is a  discrete space variable, with the axis
pointing upward.

Recall that a nonnegative random variable $X$ has Gamma distribution with shape
parameter $k>0$ and scale parameter $\theta>0$, and write $X\sim\mbox{Gamma}(k,\theta)$,
if
\[
\mathbb{P}(X\in dx)=\frac{1}{\Gamma(k)\theta^{k}}x^{k-1}e^{-x/\theta}\, dx \;.
\]
The Laplace transform of a Gamma distributed variable $X$ is given
by
\begin{equation}
\mathbb{E}[e^{tX}]=(1-\theta t)^{-k}\quad(t<1/\theta)\label{eq:Laplace-Gamma}
\end{equation}
When $k=1$, the Gamma distribution specializes to the exponential
distribution.

\begin{defn} \label{def:SWpoly}
A strict-weak polymer path $\pi$ is a lattice path which at each lattice
site $(t,n)$ is allowed to
\begin{itemize}
\item Jump horizontally to the right from $(t,n)$ to $(t+1,n)$;
\item Or, jump diagonally to the upright from $(t,n)$ to $(t+1,n+1)$.
\end{itemize}
The partition function with parameters $k,\theta>0$ for the
ensemble of strict-weak polymers from $(0,1)$ to $(t,n)$ is given
by
\[
Z(t,n)=\sum_{\pi:(0,1)\to(t,n)}\prod_{e\in\pi}d_{e}
\]
where the product is over all the horizontal and diagonal unit segments
in the path $\pi$, and
\begin{itemize}
\item $d_{e}=1$ if $e$ is a diagonal unit segment;
\item $d_{e}$ is an independent $\mbox{Gamma}(k,\theta)$ distributed random variable
if $e$ is a horizontal unit segment.
\end{itemize}
The free energy of the strict-weak polymer model is
  $\log Z(t,n)$.
\end{defn}

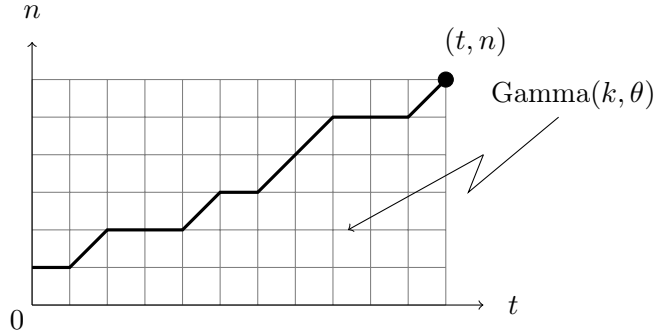
\begin{center}
\begin{tikzpicture}
\draw[step=0.5,gray,thin] (0,0) grid (5.5,3);
\draw [->] (0,0) -- (6,0);
\draw [->] (0,0) -- (0,3.5);
\node at (6.4,0) {$t$};
\node at (0,3.9) {$n$};
\filldraw (5.5,3) circle (0.1cm) ;
\node at (5.9,3.5) {$(t,n)$};
\draw [very thick] (0,0.5) -- (0.5,0.5) -- (1,1) -- (1.5,1) -- (2,1) -- (2.5,1.5) -- (3,1.5)-- (3.5,2)-- (4,2.5)-- (4.5,2.5)-- (5,2.5)-- (5.5,3);
%\filldraw (4,1) circle (0.04cm) ;
%\node at (4.23,1) {$d_e$};
\node at (-0.2,-0.2) {$0$};
\draw (7,2.5) -- (5.8,1.5) -- (6,2);
\draw [->] (6,2) -- (4.2,1);
\node at (7.2,2.8) {Gamma$(k,\theta)$};
\end{tikzpicture}
\end{center}

The partition functions of   the strict-weak polymer system satisfy
the   recursive relation
\begin{equation}
Z(t+1,n)=Y(t,n)\, Z(t,n)+Z(t,n-1)\label{eq:Recursive-Z}
\end{equation}
where $Y(t,n)$ are i.i.d.\ Gamma random variables. This relation can
be easily derived by observing that
\[
\sum_{\pi:(0,1)\to(t+1,n)}\prod_{e\in\pi}d_{e}=\sum_{\pi:(0,1)\to(t,n-1)}\prod_{e\in\pi}d_{e}\,+\, d_{f}\cdot\!\!\!\!\sum_{\pi:(0,1)\to(t,n)}\prod_{e\in\pi}d_{e}
\]
where $f$ is the horizontal edge from $(t,n)$ to $(t+1,n)$, and
therefore from the definition $d_{f}\sim\mbox{Gamma}(k,\theta)$.

The requirement that the polymer paths all start from $(0,1)$ means that
we consider the delta initial data
\begin{equation} \label{eq:initial-Z}
Z(0,n)=\mathbf{1}_{n=1} \;.
\end{equation}
Furthermore, for any point $(t,1)$ with $t\geq 0$, there is only one admissible
polymer (the straight path) from $(0,1)$ to $(t,1)$, and the total weight it collects
is the product of $t$ i.i.d.\ $\mbox{Gamma}(k,\theta)$ random variables, namely
\begin{equation} \label{eq:boundary-Z}
Z(t,1)=\prod_{s=0}^{t-1} \,d_{((s,1),(s+1,1))} \;. %Y(s,1) \;, 
%\qquad\qquad Y(s)\sim\mbox{Gamma}(k,\theta)\mbox{   i.i.d.}\;.
\end{equation}
The recursive relation \eqref{eq:Recursive-Z}, the initial condition \eqref{eq:initial-Z},
and the boundary condition \eqref{eq:boundary-Z} together determine the
partition function $Z(t,n)$ for any $t>0$ and $n>1$.
As an example,  one can see easily either from the definition or from this recursive
relation that, %for instance, $Z(1,1)\sim\mbox{Gamma}(k,\theta)$,
$Z(2,2)$ is a sum of two i.i.d $\mbox{Gamma}(k,\theta)$ random
variables, which by the property of the Gamma distribution implies that
$Z(2,2)\sim\mbox{Gamma}(2k,\theta)$.

Our main result of this paper is the KPZ universality for the strict-weak
polymer model, for sufficiently large $\kappa$ where $t=\kappa n$.
The largeness of $\kappa$ seems to be only a technical requirement to simplify the asymptotic analysis. %We also assume $\theta=1$ for simplicity.
\begin{defn} \label{def:parameters}
Recall the digamma function $\Psi(x) := \big[\log\Gamma]'(x)$. Given parameters $k>0$ and $\kappa\geq 1$ such that there exists a unique solution $\bar t \in (0,1/2)$ to the equation
\[
\Psi'(\bar{t})-\kappa\Psi'(k+\bar{t})=0  \;,
\]
we define numbers
\[
%\bar{f}_{k,\kappa}=-\Psi(\bar{t})+\kappa\Psi(k+\bar{t}) \;,
\bar{f}_{k,\theta,\kappa}=-\Psi(\bar{t})+\kappa\Psi(k+\bar{t}) +(\kappa-1)\log\theta\;,
\qquad
\bar{g}_{k,\kappa}=-\Psi''(\bar t)+\kappa\Psi''(k+\bar t) \;.
\]
\end{defn}
Lemma~\ref{lem:exists-cri} ensures that if $\kappa$ is sufficiently large, the solution $\bar t \in (0,1/2)$ exists and is unique.
Note that though $\bar{f}_{k,\theta,\kappa}$ depends on $\theta$,
$\bar{g}_{k,\kappa}$ does not depend on $\theta$ (as the notation indicates).

\begin{thm}
\label{thm:main}
There exists $\kappa^{*}=\kappa^{*}(k)>0$
such that the strict-weak polymer free energy with parameters $k,\theta>0$
%$\theta=1$, 
and $\kappa>\kappa^{*}$ has limiting fluctuation distribution
given by
\[
\lim_{n\to\infty}\mathbb{P}\bigg(
	\frac{\log Z(\kappa n,n)  %-n\bar{f}_{k,\kappa}}
			-n\bar{f}_{k,\theta,\kappa}}
	{n^{1/3}}
\le r
\bigg)
=F_{GUE}\Big(\big(\frac{\bar{g}_{k,\kappa}}{2}\big)^{-1/3}\, r\Big)
\]
where %$\bar{f}_{k,\kappa}$ 
$\bar{f}_{k,\theta,\kappa}$ 
and $\bar{g}_{k,\kappa}$ are defined in
Definition~\ref{def:parameters},
and $F_{GUE}$ is the GUE Tracy-Widom distribution function.
\end{thm}
The proof is given in Section \ref{sec:Asymptotic-analysis}.
Besides describing the fluctuations of the free energy, this theorem also proves that (in the parameter range considered) %$\bar f_{k,\kappa}$ 
$\bar f_{k,\theta,\kappa}$  represents the free energy law of large numbers.
In Section~\ref{sec:free-energy} we provide a different means (applicable for all parameter choices) to identify the free energy law of large numbers as
\[
\bar{f}_{k,\theta,\kappa}
= \inf_{\beta>0} \Big(-\Psi(\beta)+\kappa\Psi(k+\beta) +(\kappa-1)\log\theta \Big) \;.
\]
Though this appears different than the earlier expression for %$\bar{f}_{k,\kappa}$ 
$\bar{f}_{k,\theta,\kappa}$ in Definition~\ref{def:parameters}, it is readily confirmed that they are, in fact, the same.

The main observation behind the above theorem is a
connection between the strict-weak polymer and the discrete time geometric
$q$-TASEP introduced and studied in \cite{discrete-time}. Under suitable centering and scaling, the fluctuations of geometric $q$-TASEP particle positions converge weakly to the strict-weak polymer
free energies, as $q\to 1$.

Recall that the $N$-particle discrete time geometric $q$-TASEP with
jump parameter $\alpha\in(0,1)$ is an interacting particle system
with particle locations on $\Z$ labeled by
\[
X_{N}(t)< \dotsm<X_{2}(t)<X_{1}(t).
\]
In discrete time $t\in\mathbb{Z}_{\geq0}$, particles jump according
to the parallel update rule:
\[
\mathbb{P}\left(X_{n}(t+1)=X_{n}(t)+j\;\Big|\;\mbox{gap}_{n}(t)=m\right)=\mathbf{p}_{\alpha}(j\;|\; m).
\]
Here $\mbox{gap}_{n}(t):=X_{n-1}(t)-X_{n}(t)-1$ for $i>1$, and
$\mbox{gap}_{1}(t):=\infty$. The jump rates are given by
\begin{equation}
\begin{aligned}\mathbf{p}_{\alpha}(j\;|\; m) & =\alpha^{j}(\alpha;q)_{m-j}\frac{(q;q)_{m}}{(q;q)_{m-j}(q;q)_{j}},\\
\mathbf{p}_{\alpha}(j\;|\;\infty) & =\alpha^{j}(\alpha;q)_{\infty}\frac{1}{(q;q)_{j}},
\end{aligned}
\label{eq:p_alpha}
\end{equation}
where the $q$-Pochhammer symbols are defined as
\[
(a;q)_{m}=\prod_{i=0}^{m-1}(1-aq^{i})\qquad(a;q)_{\infty}=\prod_{i=0}^{\infty}(1-aq^{i}) \;.
\]
We will consider step initial condition, where, for $n\geq 1$,
\begin{equation}
X_{n}(t=0)=-n.\label{eq:step-initial}
\end{equation}

We study a particular scaling limit of the fluctuations of $X_{n}(t)$, namely the
function $F^{\eps}(t,n)$ defined via
\begin{equation}
X_{n}(t)+n=\theta^{-1}\left[ (t-(n-1))\, \eps^{-1}\log\eps^{-1}-\eps^{-1}F^{\eps}(t,n)\right]\label{eq:defFeps}
\end{equation}
under the scaling where
\begin{equation}
\alpha=e^{-m_{1}\eps},\qquad q=e^{-\theta\eps}\label{eq:scaling}.
\end{equation}

There are two ways (we know of) to motivate this scaling. The first, which is most in line with the approach we pursue herein, is that under this scaling one readily sees that the moment formulas for geometric $q$-TASEP converge to those of the strict-weak polymer (cf. the end of Section \ref{sec:Replica}). The second motivation requires a little more explanation, which we briefly describe here. Macdonald processes \cite{MR3152785} are measures on interlacing partitions which enjoy a number of exact formulas owing to the integrable structure of the Macdonald symmetric functions. A special case (corresponding to setting the Macdonald $t$ parameter to zero) yields $q$-Whittaker processes. There exist Markovian dynamics on these interlacing partitions which preserves the class of $q$-Whittaker processes, leading to a deterministic evolution on the parameters describing the fixed time marginals of the dynamics. In \cite{MR3152785} a continuous time dynamic related to the so-called Plancherel specialization is introduced, and continuous time $q$-TASEP arises as a marginal on the smallest parts of the partitions. As $q\to 1$, \cite{MR3152785} shows that the Plancherel specialized $q$-Whittaker process converges to the Plancherel Whittaker process of \cite{OCon} and $q$-TASEP converges to the free energy evolution for the O'Connell-Yor semi-discrete directed polymer. The pure alpha specialization of the $q$-Whittaker process is likewise preserved by discrete time Markov dynamics \cite{MatPet} and has discrete time geometric $q$-TASEP as its marginal on the smallest parts. The pure alpha specialized $q$-Whittaker process converges \cite{MR3152785,BCFV} (under scaling related to those above) to the alpha specialized Whittaker process of \cite{COSZ}. As explained in \cite{OConOrth}, the analog of the smallest part for the pure alpha Whittaker process is related to the strict-weak polymer free energy. Methods coming from Whittaker processes \cite{COSZ} provide a route to write down a Laplace transform formula for the strict-weak polymer partition function which can be turned (using identities similar to those of \cite{MR3116323}) into the Fredholm determinant formula present herein. This is the approach taken in \cite{OConOrth}. We do not rely upon the connection to these Macdonald/$q$-Whittaker/Whittaker processes in the approach we utilize here, though certainly this was an important motivation in our pursuit.

The following result demonstrates that the limit as $\eps\to 0$ of $e^{F^{\eps}(t,n)}$ satisfies the same recursive relation as $Z(t,n)$ where the parameter $k$ is related to $m_1$ via $k=m_{1}/\theta$. The proof is given in Section \ref{sec:Recursive-relation}, though it is also briefly sketched below.
\begin{thm}
\label{thm:rec-rel}
For $t\geq0$ and $n\geq1$, the sequence
of random variables $F^{\eps}(t,n)$ converge weakly to a limit as
$\eps\to0$, denoted as $F(t,n)$, and one has the recursive relation
\[
e^{F(t+1,n)}=Y(t,n)\, e^{F(t,n)}+e^{F(t,n-1)}
\]
for every $t\geq0$ and $n\geq1$, where $Y(t,n)$ are i.i.d.\ Gamma
distributed random variables with\textup{ }shape parameter $k=m_{1}/\theta$
and scale parameter $\theta$.
\end{thm}

Thus we see that $e^{F(t,n)}$ satisfies the same recursive relation as the polymer partition
function (\ref{eq:Recursive-Z}). When $t=0$, by step initial condition (\ref{eq:step-initial}),
we have $F^{\eps}(0,n)=(1-n)\log\eps^{-1}$, therefore $e^{F(0,n)}=\lim_{\eps\to0}e^{(1-n)\log\eps^{-1}}= \mathbf{1}_{n=1}$,
which coincides with the initial condition (\ref{eq:initial-Z}) for the polymer partition function. Also, one can show that (see Lemma~\ref{lem:first-particle})
$e^{F^\eps(1,1)}$ converges to a Gamma $(k,\theta)$ random variable. Since the first particle jumps independently at each step, $e^{F^\eps(t,1)}$ converges to a product of $t$ of i.i.d. Gamma$(k,\theta)$ random variables, so it also coincides with the boundary condition \eqref{eq:boundary-Z}.

Therefore, as a consequence of the above theorem, we obtain the convergence of the fluctuation of the geometric $q$-TASEP to the polymer free energy. In fact, the convergence of the process, or joint convergence, follows readily from the above theorem and the independence of each jump.  The independence of jumps implies independence of the random variables $Y_\eps(t,n):=(e^{F^\eps(t,n)}-e^{F^\eps(t-1,n-1)})/e^{F^\eps(t-1,n)}$, as well as independence of their limits $Y(t,n)$. Since the recursive relation is linear in these $Y_\eps(t,n)$ or $Y(t,n)$ random variables, each of the variables $e^{F^\eps(t,n)}$ or $e^{F(t,n)}=Z(t,n)$ can be written as a sum of products of different $Y_\eps$'s or $Y$'s. Consequently, weak convergence of  $\{Y_\eps(t,n)\}_{t\geq 0,n>0} \to \{Y(t,n)\}_{t\geq 0,n>0}$  implies that of the process $\{e^{F^\eps(t,n)}\}_{t\geq 0,n>0} \to \{Z(t,n)\}_{t\geq 0,n>0}$ (as can be seen, for instance, from considering characteristic functions). Summarizing, we have the following result.
\begin{cor}\label{cor:conv}
As $\eps\to0$, the   processes  $\{e^{F^\eps(t,n)}\}_{t\geq 0,n>0}$ converge  in  distribution to  the process  $\{Z(t,n)\}_{t\geq 0,n>0}$ of strict-weak polymer partition functions.
\end{cor}
%\begin{cor}\label{cor:conv}
%The sequence of random variables $\{e^{F^\eps(t,n)}\}_{t\geq 0,n>0}$ converges weakly (in terms of finite dimensional distributions) to  $\{Z(t,n)\}_{t\geq 0,n>0}$.
%\end{cor}

Given this convergence result, we can apply the exact formula for the $e_q$-Laplace transform of the particle location fluctuations of the geometric
$q$-TASEP to obtain an exact formula for the strict-weak polymer. The following Fredholm determinant formula for the geometric $q$-TASEP is from \cite[Theorem~2.4]{discrete-time}.
\begin{thm}
For every $\zeta\in\mathbb{C}\backslash\mathbb{R}_{+}$,
\begin{equation}
\mathbb{E}\bigg[\frac{1}{(\zeta q^{X_{n}(t)+n};q)_{\infty}}\bigg]=\det(I+K_{\zeta})_{L^{2}(C_{1})}\label{eq:q-Laplace}
\end{equation}
where  $C_{1}$ is a small positively oriented circle containing $1$ and
$K_{\zeta}:L^{2}(C_{1})\to L^{2}(C_{1})$   is given by its integral kernel
\[
K_{\zeta}(w,w')=\frac{1}{2\pi\I}\int_{-\I\infty+1/2}^{\I\infty+1/2}\frac{\pi}{\sin(-\pi s)}(-\zeta)^{s}\frac{(q^{s}w;q)_{\infty}^{n}}{(w;q)_{\infty}^{n}}\frac{(\alpha w;q)_{\infty}^{t}}{(\alpha q^{s}w;q)_{\infty}^{t}}\frac{1}{q^{s}w-w'}ds \;.
\]
\end{thm}

From the above formula, we take the $q\to1$ limit according to the
scaling (\ref{eq:defFeps}) and (\ref{eq:scaling}) and obtain the following Fredholm determinant
formula for strict-weak polymers; the proof of the following formula
is given in Section \ref{sec:Fredholm-SWpolymer}.
\begin{thm}
\label{thm:Fredholm-SWpolymer}For $u\in\mathbb{C}$ such that $Re(u)>0$,
let $t=\kappa n$ for parameter $\kappa\geq1$. Then one has
\[
\mathbb{E}\Big[e^{-uZ(\kappa n,n)}\Big]=\det(I+K_{u})_{L^{2}(C_{0})}
\]
where $C_{0}$ is a small positively
oriented circle containing $0$ and
$K_{u}:L^{2}(C_{0})\to L^{2}(C_{0})$ has kernel
\[
K_{u}(v,v')=
\frac{1}{2\pi\I}\int_{-\I\infty+1/2}^{\I\infty+1/2}
\frac{\pi}{\sin(\pi(v-\tilde{z}))}\Big(\frac{\Gamma(v)/\Gamma(k+v)^{\kappa}}{\Gamma(\tilde{z})/\Gamma(k+\tilde{z})^{\kappa}}\Big)^{n}
\frac{u^{\tilde{z}-v}\theta^{(\kappa n-(n-1))(\tilde{z}-v)}}{\tilde{z}-v'} d\tilde{z} \;.
\]
\end{thm}
We use this Fredholm determinant formula to prove Theorem \ref{thm:main}.

We remark that there is a zero-temperature limit of our model as $k\to 0$ previously studied 
in \cite{Neil1999directed}. In fact as $k\to 0$, the family of random variables
$-k\log d_e$ converge to a family of independent exponential   random variables, and the model converges weakly to a directed first passage percolation model
(i.e. a problem of minimizing the total   weights along paths).

%\medskip

\subsection{Outline}
Section \ref{sec:Recursive-relation} contains the proof of Theorem \ref{thm:rec-rel}. In Section \ref{sec:Fredholm-SWpolymer} we prove Theorem \ref{thm:Fredholm-SWpolymer}. In Section
\ref{sec:Asymptotic-analysis} we carry out rigorous asymptotic analysis based on the formula in Theorem \ref{thm:Fredholm-SWpolymer} and prove Theorem \ref{thm:main}. In Section \ref{sec:Replica} we
apply the replica method to derive moments formula of the polymer partition function. Finally in Section~\ref{sec:stat} we introduce a stationary version
of the polymer model and in Section~\ref{sec:free-energy} we identify the free energy law of large numbers using this stationary model.

\subsection{Acknowledgements}
I. Corwin was partially supported by the NSF grant DMS-1208998 as well as by Microsoft Research and MIT through the Schramm Memorial Fellowship, by the Clay Mathematics Institute through the Clay Research Fellowship and by the Institut Henri Poincar\'e through the Poincar\'e Chair. H. Shen would like to thank Prof. Martin Hairer for his support on a visit to MSRI in July 2014 where part of this work was done. T.~Sepp\"al\"ainen was partially supported by  NSF grant DMS-1306777 and by the Wisconsin Alumni Research Foundation.

\section{Recursive relation: Proof of Theorem \ref{thm:rec-rel}\label{sec:Recursive-relation}}

The proof of Theorem \ref{thm:rec-rel} follows from the definition of the discrete time geometric $q$-TASEP and certain known limits of $q$-deformed functions. We will first demonstrate the limit of the fluctuation of the first particle.

\begin{lem} \label{lem:first-particle}
The sequence of random variables $\exp(F^{\eps}(1,1))$ converge as $\eps\to0$ to a Gamma
distributed random variable with shape parameter $k=m_{1}/\theta$ and scale parameter $\theta$.
\end{lem}

\begin{proof}
By the definition \eqref{eq:defFeps} of the
quantity $F^{\eps}(1,1)$,
for any positive real number $r$, one has $e^{F^{\eps}(1,1)} =r$
if and only if
\be \label{eq:int-int}
X_{1}(1)+1=\theta^{-1}\left[\eps^{-1}\log\eps^{-1}-\eps^{-1}\log r\right] \;.
\ee
Since the left side above is always a non-negative integer,
$F^{\eps}(1,1)$ can only take values $r$ in a discrete set
such that the right side above is also a non-negative integer,
namely $\log r\in \log\eps^{-1} -\eps\theta \,\Z_+$.
%the number of  $r\in[r_0,r_1)$ such that the right side above is also
%an integer is equal to $(\theta\eps)^{-1} \log\frac{r_1}{r_0}$.
For every such $r$, by the definition
of the discrete time geometric $q$-TASEP,
%$\mathbb{P}\Big(  e^{F^{\eps}(1,1)}\in[r_0,r_1]\Big)$
%is equal to the Jacobian arising from transformation $F^\eps(1,1) \to X_1(1)$
%which is $(\theta\eps r)^{-1}$, multiplied by
\begin{equation}
\begin{aligned}
\mathbb{P}\Big(X_{1}(1) & +1  =\theta^{-1}\left[\eps^{-1}\log\eps^{-1}-\eps^{-1}\log r\right]\Big)\\
 & =\mathbf{p}_{\alpha}(\theta^{-1}\left[\eps^{-1}\log\eps^{-1}-\eps^{-1}\log r\right]\;|\;\infty)\\
 & =e^{-\eps m_{1}\theta^{-1}\left[\eps^{-1}\log\eps^{-1}-\eps^{-1}\log r\right]}\frac{(e^{-\eps m_{1}};e^{-\eps\theta})_{\infty}}{(e^{-\eps\theta};e^{-\eps\theta})_{\theta^{-1}\left[\eps^{-1}\log\eps^{-1}-\eps^{-1}\log r\right]}}
 \;,
\end{aligned}
\label{eq:ProbF11}
\end{equation}
where $\mathbf{p}_{\alpha}$
is defined in \eqref{eq:p_alpha}.
The exponential factor
\[
e^{-\eps m_{1}\theta^{-1}\left[\eps^{-1}\log\eps^{-1}-\eps^{-1}\log r\right]}
=
\eps^{m_{1}/\theta}r^{m_{1}/\theta}\;.
\]
By \cite[Corollary 4.1.10]{MR3152785}, if we define
\[
f(y,\eps)=(e^{-\eps},e^{-\eps})_{\eps^{-1}\log\eps^{-1}+\eps^{-1}y} \;,
\]
then for any $\delta>0$, there exists $\eps_{0}>0$ such that if
$\eps<\eps_{0}$ one has
\be \label{eq:log-f}
\log\, f(y,\eps)-\mathcal{A}(\eps)-e^{-y}\in[-\delta,\delta]
\ee
where $\mathcal{A}(\eps)$ is an $\eps$ dependent constant (whose value
is not important since in our case it will cancel out).
Note that in our case,
\[
(e^{-\eps\theta};e^{-\eps\theta})_{\theta^{-1}\left[\eps^{-1}\log\eps^{-1}-\eps^{-1}\log r\right]}=f(-\log r-\log\theta^{-1};\,\eps\theta) \;.
\]
Therefore for any $\delta>0$, if $\eps$ is sufficiently small
\begin{equation}
(e^{-\eps\theta};e^{-\eps\theta})_{\theta^{-1}\left[\eps^{-1}\log\eps^{-1}-\eps^{-1}\log r\right]}\cdot e^{-\mathcal{A}(\eps\theta)-r/\theta}\in[1-\delta,1+\delta]\label{eq:theta-theta-r} \;.
\end{equation}

As for the $r$-independent factor in the numerator (which will be
a normalization factor), by the definition of $q$-Gamma function
\be \label{eq:q-gamma}
\Gamma_{q}(x):=\frac{(q;q)_{\infty}}{(q^{x};q)_{\infty}}(1-q)^{1-x} \;,
\ee
we can take  $x=m_{1}/\theta$, so that
\begin{equation} \label{eq:m1-theta-inf}
(e^{-\eps m_{1}};e^{-\eps\theta})_{\infty}
=\frac{(e^{-\eps\theta};e^{-\eps\theta})_{\infty}}
	{\Gamma_{e^{-\eps\theta}}(m_{1}/\theta)}
(1-e^{-\eps\theta})^{1-m_{1}/\theta} \;.
\end{equation}
And one has
\[
(e^{-\eps\theta};e^{-\eps\theta})_{\infty}=f(\infty;\,\eps\theta)
\]
and therefore for $\eps$ sufficiently small
\[
(e^{-\eps\theta};e^{-\eps\theta})_{\infty}\cdot e^{-\mathcal{A}(\eps\theta)}\in[1-\delta,1+\delta] \;.
\]

Substitute (\ref{eq:theta-theta-r}) and (\ref{eq:m1-theta-inf})
into (\ref{eq:ProbF11}), and we obtain that %for any $\delta>0$,
the quantity %$\mathbb{P}\Big(e^{F^{\eps}(1,1)}\in[r_0,r_1)\Big)$
\eqref{eq:ProbF11}
 is arbitrarily close to
\[
%\sum_{r}
%\Big( (\eps\theta)^{-1} \log\frac{r_1}{r_0} \Big)
\eps^{m_{1}/\theta}r^{m_{1}/\theta}\frac{(1-e^{-\eps\theta})^{1-m_{1}/\theta}}{\Gamma_{e^{-\eps\theta}}(m_{1}/\theta)}e^{-r/\theta}
%=: r\theta \cdot p_\eps(r)
\]
for $\eps$ sufficiently small.
%where the sum is over  $r\in[r_0,r_1)$ such that the right side of \eqref{eq:int-int}
%is integer.

In general, if $F^\eps(1,1)$ is a random variable valued in $a_\eps + \eps\theta \Z$,
where $a_\eps$ is an $\eps$ dependent shift,
and for any $s\in\R$, one has $(\eps\theta)^{-1}\mathbb P(F_\eps(1,1)=\hat s)\to f(s)$ as
$\eps\to 0$ where $\hat s =\max\{s'\le s| s'\in a_\eps + \eps\theta \Z\}$,
then $F^\eps(1,1)$ converges to a limit $F$ as $\eps\to 0$ weakly and $F$
takes value in the continuum and has $f$ as its density function.
This can be proved, for instance, via approximating $\mathbb P(F_\eps(1,1)>t)$
by $\int_t^{\infty} (\theta\eps)^{-1} \mathbb P(F_\eps(1,1)=\hat s)\,ds$
up to a small error which goes to $0$ as $\eps\to 0$.
This integral converges to $\int_t^\infty f(s)\,ds$
by point-wise convergence and applying Fatou's lemma on both $[t,\infty)$
and $(-\infty,t]$, and the fact that a density function integrates to $1$
over $(-\infty,\infty)$.

In our case, note that $(1-e^{-\eps\theta})/(\eps\theta)\to1$
as $\eps\to0$, and that $k=m_1/\theta$.
Therefore  for any positive real number $r$, letting $s=\log r$,
\[
(\eps\theta)^{-1}\mathbb P(F_\eps(1,1)=\hat s)  \to
r \cdot \frac{r^{k-1}}{\theta^{k} \Gamma(k)}e^{-r/\theta} =:f(s)
\qquad (\mbox{ as }\eps \to 0,\mbox{  where }r=e^s).
\]
So $F^\eps(1,1)$ converges to a limiting random variable $F$
and its density function $\mathbb P(F\in[s,s+ds))$
is equal to $f(s)\,ds$ with $f$ defined above.
Since $ds=\frac{1}{r}dr$, one concludes that
$e^{F^\eps(1,1)}$ converges weakly to $e^{F(1,1)}$
which is a Gamma$(k,\theta)$ distributed random variable.
%
%taking $r_1=r_0+dr$ such that the number of terms in the above sum
%$(\theta\eps)^{-1} \log\frac{r_1}{r_0}
%=(\theta\eps r_0)^{-1} dr$
%one has (not quite right...)
%\[
%\mathbb{P}\Big(e^{F^{\eps}(1,1)}\in[r_0,r_0+dr)\Big)
%\to\frac{r_0^{k-1}}{\theta^k \Gamma(k)}e^{-r_0/\theta} dr
%\qquad
%(\eps\to 0)
%\]
%which is the Gamma distribution.
\end{proof}

Since the geometric $q$-TASEP is defined
in terms of the probability of the distance that the $n$-th particle jumps forward from time $t-1$ to time $t$, given the gap between
the $n$-th particle and the $(n-1)$-st particle at time $t-1$, it
is natural to consider the distribution of $F^{\eps}(t,n)$ given the
values of $F^{\eps}(t-1,n)$ and $F^{\eps}(t-1,n-1)$. This motivates
the following proof.
\begin{proof}[Proof of Theorem~\ref{thm:rec-rel}]
We compute the probability that $e^{F^{\eps}(t,n)}=r$ conditioned
on $e^{F^{\eps}(t-1,n)}=v$ and $e^{F^{\eps}(t-1,n-1)}=u$.
%In this computation we will make an approximation and drop the shifts by $n$ of $X_n(t)$. This introduces an order $\epsilon$ correction which is inconsequential in the limit $\eps\to 0$. With this approximation,
Observe that we seek to study the probability that
\[
X_{n}(t) + n =\theta^{-1}\left[(t-(n-1)) \,\eps^{-1}\log\eps^{-1}-\eps^{-1}\log r\right]
\]
conditioned on
\[
X_{n-1}(t-1) + (n-1) =\theta^{-1}\left[(t-(n-1)) \,\eps^{-1}\log\eps^{-1}-\eps^{-1}\log u\right] \;,
\]
\[
X_{n}(t-1) + n =\theta^{-1}\left[(t-n)\,\eps^{-1}\log\eps^{-1}-\eps^{-1}\log v\right] \;,
\]
where $r$, $u$ and $v$ take discrete values  such that the right hand sides
of the above identities are integers.
This means that at time $t-1$, the gap between the $(n-1)$-st and
$n$-th particle is given by
\[
X_{n-1}(t-1)-X_{n}(t-1)
=\theta^{-1}\left[\eps^{-1}\log\eps^{-1}-\eps^{-1}\log(u/v)\right] + 1
\]
and one asks for the probability that the $n$-th particle jumps by
the distance
\begin{equation}
X_{n}(t)-X_{n}(t-1)=\theta^{-1}\left[\eps^{-1}\log\eps^{-1}-\eps^{-1}\log(r/v)\right]\label{eq:jump-rv}
\end{equation}
Therefore, the conditional probability
\[
\mathbb{P}\Big(e^{F^{\eps}(t,n)}=r\,\Big|\, e^{F^{\eps}(t-1,n-1)}=u,\, e^{F^{\eps}(t-1,n)}=v\Big)
\]
is equal to %the Jacobian $v/(\theta\eps r)$ arising from the transformation
%(\ref{eq:jump-rv}) multiplied by the following function
(the jump rates for q-TASEP $\mathbf{p}_{\alpha}$
is defined in (\ref{eq:p_alpha})):
\begin{equation}
\begin{aligned} & \mathbf{p}_{\alpha}\Big(
	\theta^{-1}\left[\eps^{-1}\log\eps^{-1}-\eps^{-1}\log\frac{r}{v}\right]\,\Big|\,\theta^{-1}\left[\eps^{-1}\log\eps^{-1}-\eps^{-1}\log\frac{u}{v}\right] +1
	\Big)\\
 & =e^{-\eps m_{1}\cdot\theta^{-1}\left[\eps^{-1}\log\eps^{-1}-\eps^{-1}\log\frac{r}{v}\right]}\,
 (e^{-\eps m_{1}};e^{-\eps\theta})_{\theta^{-1}\left[\eps^{-1}\log\frac{r}{v}-\eps^{-1}\log\frac{u}{v}\right] +1}\\
 & \quad\times
 \frac{(e^{-\eps\theta};e^{-\eps\theta})_{\theta^{-1}\left[\eps^{-1}\log\eps^{-1}-\eps^{-1}\log(u/v)\right] +1}}{(e^{-\eps\theta};e^{-\eps\theta})_{\theta^{-1}\left[\eps^{-1}\log\frac{r}{v}-\eps^{-1}\log\frac{u}{v}\right] +1}(e^{-\eps\theta};e^{-\eps\theta})_{\theta^{-1}\left[\eps^{-1}\log\eps^{-1}-\eps^{-1}\log\frac{r}{v}\right]}}
\end{aligned}
\label{eq:ruv}
\end{equation}
For the exponential factor, one has (recall that $k=m_{1}/\theta$)
\[
\lim_{\eps\to0}\,\eps^{-k}e^{-\eps m_{1}\cdot\theta^{-1}\left[\eps^{-1}\log\eps^{-1}-\eps^{-1}\log(r/v)\right]}=(r/v)^{k} \;.
\]

As in the proof of Lemma~\ref{lem:first-particle},
By \cite[Corollary 4.1.10]{MR3152785}, if we define
\[
f(y,\eps)=(e^{-\eps};e^{-\eps})_{[\eps^{-1}\log\eps^{-1}+\eps^{-1}y]} \;,
\]
then for any $\delta>0$, if
$\eps$ is sufficiently small then one has \eqref{eq:log-f}.
%\[
%\log\, f(y,\eps)-\mathcal{A}(\eps)-e^{-y}\in[-\delta,\delta]
%\]
%where $\mathcal{A}(\eps)$ is an $\eps$ dependent constant (whose value
%is not important since in our case it will cancel out).
Using this fact, some of the factors in (\ref{eq:ruv}) can be written as
\[
(e^{-\eps\theta};e^{-\eps\theta})_{\theta^{-1}\left[\eps^{-1}\log\eps^{-1}-\eps^{-1}\log(r/v)\right]}=f(-\log(r/v)-\log\theta^{-1};\,\eps\theta) \;.
\]
Therefore for any $\delta>0$, if $\eps$ is sufficiently small
\[
(e^{-\eps\theta};e^{-\eps\theta})_{\theta^{-1}\left[\eps^{-1}\log\eps^{-1}-\eps^{-1}\log(r/v)\right]}\cdot e^{-\mathcal{A}(\eps\theta)-r/(v\theta)}\in[1-\delta,1+\delta] \;.
\]
Similarly, one has (now with $y=-\log(u/v)-\log\theta^{-1} +\eps\theta$)
\[
(e^{-\eps\theta};e^{-\eps\theta})_{\theta^{-1}\left[\eps^{-1}\log\eps^{-1}-\eps^{-1}\log(u/v)\right] +1}
\cdot e^{-\mathcal{A}(\eps\theta)-u/(v\theta e^{\eps\theta})}\in[1-\delta,1+\delta] \;.
\]

For the other two factors in (\ref{eq:ruv}), one has
\begin{equation}
\begin{aligned} & \frac{
(e^{-\eps m_{1}},e^{-\eps\theta})_{\theta^{-1}\left[\eps^{-1}\log(r/v)-\eps^{-1}\log(u/v)\right] +1}
}{(e^{-\eps\theta},e^{-\eps\theta})_{\theta^{-1}\left[\eps^{-1}\log(r/v)-\eps^{-1}\log(u/v)\right] +1}}\\
 & \qquad=
 \frac{(e^{-\eps\theta},e^{-\eps\theta})_{m_{1}/\theta+\theta^{-1}\left[\eps^{-1}\log(r/v)-\eps^{-1}\log(u/v)\right]}}{
 (e^{-\eps\theta},e^{-\eps\theta})_{\theta^{-1}\left[\eps^{-1}\log(r/v)-\eps^{-1}\log(u/v)\right] +1}
 (e^{-\eps\theta},e^{-\eps\theta})_{m_{1}/\theta-1}}
\end{aligned}
\label{eq:ratio}
\end{equation}
The factor $(e^{-\eps\theta},e^{-\eps\theta})_{m_{1}/\theta-1}$ in
the denominator will only contribute as a normalization factor. To
compute it, we use the $q$-Gamma function (\ref{eq:q-gamma}).
With $q=e^{-\eps\theta}$ and $x=m_{1}/\theta$, we have
\[
(e^{-\eps\theta};e^{-\eps\theta})_{\frac{m_{1}}{\theta}-1}=\Gamma_{e^{\eps\theta}}(m_{1}/\theta)\,(1-e^{-\eps\theta})^{m_{1}/\theta-1} \;.
\]
Therefore,
\[
\lim_{\eps\to0}\,(\eps\theta)^{1-m_{1}/\theta}(e^{-\eps\theta};e^{-\eps\theta})_{\frac{m_{1}}{\theta}-1}=\Gamma(k) \;.
\]
By definition, the ratio of the other two factors in (\ref{eq:ratio})
is
\[
\prod_{i=2}^{m_{1}/\theta}
(1-(e^{-\eps\theta})^{\theta^{-1}\left[\eps^{-1}\log(r/v)-\eps^{-1}\log(u/v)\right]+i})\to(1-u/r)^{\frac{m_{1}}{\theta}-1} \;.
\]

Note that the set of admissible values of the conditioning variables $u,v$
also depends on $\eps$, that is, $\log u,\log v\in (t-(n-1))\log\eps^{-1}-\theta\eps\Z_+$. In the interval $[u,ue^{\theta\eps})$ there is only one admissible value of $u$, and similarly for $v$. This  implies (combining the above analysis together)
\be \label{eq:cond-gamma-rv}
\begin{aligned}
\lim_{\eps\to0}\, & (\theta\eps)^{-1}\mathbb{P}\Big(e^{F^{\eps}(t,n)}=r\,\Big|\, e^{F^{\eps}(t-1,n-1)}\in[u,ue^{\theta\eps}),\, e^{F^{\eps}(t-1,n)}\in[v,ve^{\theta\eps})\Big)\\
 & = %\Big(\frac{v}{\theta r}\Big)\cdot
 \frac{(r/v)^{k}\,(1-u/r)^{k-1}}{\Gamma(k)\theta^{k}}\, e^{-(r-u)/(v\theta)}\\
 & =\frac{r}{v}\cdot(\frac{r-u}{v})^{k-1}\, e^{-(r-u)/(v\theta)}\Big/\Big(\Gamma(k)\theta^{k}\Big) \;.
\end{aligned}
\ee

Now we follow the same argument as in the proof of  Lemma \ref{lem:first-particle} about convergence of discrete valued random variables $F^\eps$ to continuum valued random variable $F$.
This gives the conditional probability of $F(t,n)=s:=\log r$, conditioned on $F(t-1,n-1)=\log u$ and $F(t-1,n)=\log v$. Let $w=(e^s-u)/v$, then $ds=\frac{v}{r}dw$. Note that this factor $\frac{v}{r}$ cancels with the factor $\frac{r}{v}$ in the last line of \eqref{eq:cond-gamma-rv}. Therefore, we have that $F^{\eps}(t,n)\to F(t,n)$ and that $(e^{F(t,n)}-e^{F(t-1,n-1)})/e^{F(t-1,n)}$
are $\mbox{Gamma}(k,\theta)$ distributed. The recursive relation follows immediately.
\end{proof}

\section{Strict-weak Fredholm determinant formula: Proof of Theorem \ref{thm:Fredholm-SWpolymer}}\label{sec:Fredholm-SWpolymer}

We prove the Fredholm determinant formula in Theorem \ref{thm:Fredholm-SWpolymer} for the Laplace transform of the polymer partition function. Firstly, we show that under proper scalings, the left hand side of
(\ref{eq:q-Laplace}) goes to the Laplace transform of $Z(t,n)=e^{F(t,n)}$.
We scale the parameter as
\[
\zeta=-\eps^{n-t}\theta u
\]
and scale other parameters as in (\ref{eq:defFeps}) and (\ref{eq:scaling}). Then
we have
\[
\mathbb{E}\bigg[\frac{1}{(\zeta q^{X_{n}(t)+n};q)_{\infty}}\bigg]=\mathbb{E}\big[e_{q}(x_{q})\big]
\]
where
\[
e_{q}(x)=\frac{1}{\big((1-q)x;q\big)_{\infty}}
\]
is the $q$-exponential, and
\[
x_{q}=-\frac{\eps\theta}{1-q}\, u\, e^{F^{\eps}(t,n)}  \;.
\]
Therefore noticing that $e_{q}(x)\to e^{x}$ uniformly and $\frac{\eps\theta}{1-q}\to1$
as $\eps\to 0$ and $q\to 1$ under the scaling \eqref{eq:scaling},
we have, by Lemma 4.140 of \cite{MR3152785} along with the convergence result of Corollary \ref{cor:conv}, that
\[
\lim_{\eps\to0}\mathbb{E}\Big[\frac{1}{(\zeta q^{X_{n}(t)+n};q)_{\infty}}\Big]=\mathbb{E}\big[\exp(-ue^{F(t,n)})\big] \;.
\]

As the next step, we study the limit of $K_{\zeta}$ from (\ref{eq:q-Laplace}) with
\[
\zeta=-\eps^{n-t}\theta u,\qquad w=q^{v},\qquad w'=q^{v'} \;.
\]
At first, we will not take care of describing contours and will only discuss pointwise convergence of the integrand.

Recalling the $q$-Gamma function from (\ref{eq:q-gamma}), we can write
\[
\frac{(q^{s}w;q)_{\infty}^{n}}{(w;q)_{\infty}^{n}}=\Big(\frac{\Gamma_{q}(v)}{\Gamma_{q}(s+v)}\frac{1}{(1-q)^{s}}\Big)^{n} \;,
\]
\begin{equation} \label{eq:willuse}
\frac{(\alpha w;q)_{\infty}^{t}}{(\alpha q^{s}w;q)_{\infty}^{t}}=\Big((1-q)^{s}\frac{\Gamma_{q}(\frac{m_{1}}{\theta}+s+v)}{\Gamma_{q}(\frac{m_{1}}{\theta}+v)}\Big)^{t} \;.
\end{equation}
Combining the above expressions, as well as noting the Jacobian factor $\tfrac{dw}{dv} = q^v \log q$, we find that $\det(1+K_{\zeta})_{L^2(C_1)} = \det(1+\tilde K_{\zeta})_{L^2(C_0)}$ where $C_0$ is a small circle around the origin and the kernel $\tilde K_{\zeta}$ is defined as
\[
\tilde{K}_{\zeta}(v,v') =\frac{1}{2\pi\I}\int_{-\I\infty+1/2}^{\I\infty+1/2}\, h^{q}(s)\, ds
\]
where
\[h^{q}(s) =\frac{\pi}{\sin(-\pi s)} (-\zeta)^{s}
\Big(\frac{\Gamma_{q}(v)}{\Gamma_{q}(s+v)}\frac{1}{(1-q)^{s}}\Big)^{n}\Big((1-q)^{s}\frac{\Gamma_{q}(\frac{m_{1}}{\theta}+s+v)}{\Gamma_{q}(\frac{m_{1}}{\theta}+v)}\Big)^{t}
\frac{q^v \log q}{q^{s}q^{v}-q^{v'}}.
\]

As $\eps\to0$, observe that
\begin{align}
\eps^{ts}\Big(\frac{-\zeta}{(1-q)^{n}}\Big)^{s}&\to(u/\theta^{n-1})^{s} \;,\\
\frac{q^{v}\log q}{q^{s+v}-q^{v'}}&\to\frac{1}{s+v-v'} \;,\\
\eps^{-ts}\bigg((1-q)^{s}\frac{\Gamma_{q}(\frac{m_{1}}{\theta}+s+v)}{\Gamma_{q}(\frac{m_{1}}{\theta}+v)}\bigg)^{t}&\to\theta^{st}\bigg(\frac{\Gamma(\frac{m_{1}}{\theta}+s+v)}{\Gamma(\frac{m_{1}}{\theta}+v)}\bigg)^{t} \;,\\
\bigg(\frac{\Gamma_{q}(v)}{\Gamma_{q}(s+v)}\bigg)^{n}&\to\bigg(\frac{\Gamma(v)}{\Gamma(s+v)}\bigg)^{n} \;.
\end{align}
Letting $\tilde{z}=s+v$, and $t=\kappa n$, the above considerations suggest that $\tilde{K}_{\zeta}(v,v')$ converges to
\begin{equation}
K_{u}(v,v')=\frac{1}{2\pi\I}\int_{-\I\infty+1/2}^{\I\infty+1/2}\,\frac{\pi}{\sin(\pi(v-\tilde{z}))}\frac{F(\tilde{z})}{F(v)}\frac{1}{\tilde{z}-v'}\bigg(\frac{\Gamma(v)}{\Gamma(\tilde{z})}\bigg)^{n}d\tilde{z}\label{eq:Ku}
\end{equation}
where
\[
F(z)=u^{z}\,\theta^{(\kappa n-(n-1)) \,z}\,\Gamma\Big(\frac{m_{1}}{\theta}+z\Big)^{\kappa n} \;.
\]
%Note that the poles of $F$ are all on the negative real line. We can deform the contours such that $v,v'$ are points on a small circle around the origin, and $\tilde{z}$ is integrated along the straight line $-\frac{1}{2}+\I\mathbb{R}$.

If we can suitably strengthen the above pointwise convergence of the integrand of the kernel then we can deduce the convergence of the associated Fredholm determinants $\det(1+K_{\zeta})$ to $\det(1+K_{u})$. The proof of this convergence which we provide now is analogous to that in \cite{MR3152785}. First of all, note that for any fixed compact subset $D$ of $-\frac{1}{2}+\I\mathbb{R}$,
the convergence of the integrand of $\tilde K_{\zeta}$ is uniform over $s\in D$.
This is due to the fact that the $\Gamma_{q}$ function converges uniformly
to the $\Gamma$ function on compact domains away from poles (the terms are easily seen to satisfy uniform convergence as well).

The following tail bounds shows that the integrals in $s$ variables
in the Fredholm expansion can be restricted to compact sets, as the
contribution to the integrals from outside these compact sets can
be  bounded (uniformly in $q$ near 1) arbitrarily close to zero by choosing large enough compact
sets.
\begin{lem}
\label{lem:tail-hq}
Let $D\subset \mathbb{C}$ be an arbitrary compact set containing the unit disk
$\{z:|z|<1\}$. For all $\kappa\geq1$, one has the following
tail bound of $h^{q}(s)$: there exists positive constants $C,c$
such that for all $s\in (-\frac{1}{2}+\I\mathbb{R})\backslash D$, all $q\in(1/2,1)$,
and all $v,v'\in C_{0}$, the following bound holds
\[
|h^{q}(s)| \leq Ce^{-c\,|Im(s)|} \;.
\]
\end{lem}
\begin{proof}
The factor $\left|\frac{\pi}{\sin(-\pi s)}\right|$ decays exponentially in $|Im(s)|$.
The factors $(-\zeta)^{s},\Gamma_{q}(v),(1-q)^{\pm s},\Gamma_{q}(\frac{m_{1}}{\theta}+v)^{-1}$, and
$\frac{q^v \log q}{q^{s}q^{v}-q^{v'}}$ and $\Gamma_{q}(\frac{m_{1}}{\theta}+s+v)^{\kappa-1}$
can be all bounded by constants independent of $q$ and $|Im(s)|$.
Therefore we only need to bound the quantity
\[
\Big(\frac{1}{\Gamma_{q}(s+v)}\Big)^{n}
\Big(\Gamma_q(\frac{m_{1}}{\theta}+s+v)\Big)^{t}
=
\Big(\frac{\Gamma_q(\frac{m_{1}}{\theta}+s+v)}{ \Gamma_{q}(s+v)}
\cdot \Gamma_q(\frac{m_{1}}{\theta}+s+v)^{\kappa-1}\Big)^n  \;.
\]
Using the assumption $\kappa\geq 1$, one has
\[
%\max\Big(
\Big| \Gamma_q(\frac{m_{1}}{\theta}+s+v)^{\kappa-1} \Big| 
%\Big| \frac{\Gamma_{q}(\frac{m_{1}}{\theta}+s+v)}{\Gamma_{q}(s+v)} \Big|
%\Big)
<C' \;,
\qquad s\in (-\frac{1}{2}+\I\mathbb{R})\backslash D
\]
for a constant $C'$ independent of $q$ and $|Im(s)|$. 
Furthermore, writing $s=-\frac{1}{2}+iy$ and using $\Gamma_q (x+1)=[x]_q \Gamma_q(x)$ where $[x]_q$ is the q-number, one has
\[
\Big| \frac{\Gamma_{q}(\frac{m_{1}}{\theta}+s+v)}{\Gamma_{q}(s+v)} \Big|
= \Big| \frac{
	\Gamma_{q}(\frac{m_{1}}{\theta}+\frac{1}{2}+v+iy) \, [-\frac{1}{2}+v +iy]_q}
	{\Gamma_{q}(\frac{1}{2}+v+iy) \, [\frac{m_{1}}{\theta}-\frac{1}{2}+v +iy]_q } \Big|
	\;.
\]
For $s\in (-\frac{1}{2}+\I\mathbb{R})\backslash D$,
the norm of the ratio of the two q-numbers is bounded by a constant 
uniformly in $q \in(1/2,1)$. 
Furthermore, since $v\in C_0$ and $C_0$ is a sufficiently small circle around the origin,
$\frac{m_{1}}{\theta}+\frac{1}{2}+Re(v)$ and $\frac{1}{2}+Re(v)$
are positive. Therefore we can apply  \cite[Lemma~2.7]{Barraquand2015}, which states that there exists a constant $C''>0$ such that for all $s\in (-\frac{1}{2}+\I\mathbb{R})\backslash D$,
\[
\Big| \frac{
	\Gamma_{q}(\frac{m_{1}}{\theta}+\frac{1}{2}+v+iy)} 	{\Gamma_{q}(\frac{1}{2}+v+iy)  } \Big| \le C''(|y|^{\frac{m_{1}}{\theta}+1} +1) 
	\;.
\]
Since this polynomially growing bound is dominated by
 the exponential decaying factor mentioned in the beginning of the proof,
 the desired tail bound holds.
\end{proof}
The condition $\kappa\geq 1$ of the previous lemma is only a very tiny restriction. In fact the partition function is zero for $t<n-1$ by definition of the allowed polymer paths.

The following result together with Hadamard\textquoteright{}s bound
shows that it suffices to consider only a finite number of terms in
the Fredholm expansion, as the contribution of the later terms can
be bounded  (uniformly in $q$ near 1) arbitrarily close to zero by going out far enough in the
expansion.
\begin{lem}
There exists a constant $C>0$ such that for all $q\in(1/2,1)$, and
all $v,v'\in C_{0}$ one has $|\tilde{K}_{\zeta}(v,v')|<C$.\end{lem}
\begin{proof}
For any compact domain $D$, by Lemma \ref{lem:tail-hq} the integration
over $s$ outside $D$ can be bounded uniformly in $q,v,v'$. Inside
$D$ we use the uniform convergence and the fact that $h^q(s)$ is bounded uniformly in $q$ and $v,v'\in C_0$.
\end{proof}
It is now standard to combine the above estimates to show convergence of the Fredholm determinant expansions. The boundedness of $\tilde{K}_{\zeta}$ (as well as $K_{u}$), compactness of the contour $C_0$, and Hadamard's inequality enables us to cut off the Fredholm determinant expansions after a finite number of terms with small error (going to zero as the number of terms increases). Then, using the exponential decay of $h^q(s)$ as well as its uniform convergence to its pointwise limit, we arrive at the convergence of these finite Fredholm expansion terms to their limiting analogs, thus completing the proof of the theorem.

\section{Asymptotic analysis: proof of Theorem \ref{thm:main}\label{sec:Asymptotic-analysis}}

We start by observing that a suitable limit of the Laplace
transform of $Z(\kappa n,n)$ will give the asymptotic probability distribution of $\log Z(\kappa n,n)$, centered and scaled. We then apply the same limit to the Fredholm determinant formula proved earlier for $Z(\kappa n,n)$. Overall, the proof of Theorem \ref{thm:main} follows a similar line as in \cite{MR3116323}.

Let
\[
u=u(n,r,k,\theta,\kappa):=e^{-n\bar{f}_{k,\theta,\kappa}-rn^{1/3}} \;,
\]
where $\bar{f}_{k,\theta,\kappa}$ will be specified later.
If for each $r\in\mathbb{R}$ we have
\[
\lim_{n\to\infty}\mathbb{E}e^{-uZ(\kappa n,n)}=p_{k,\theta,\kappa}(r)
\]
where $p_{k,\theta,\kappa}(r)$ is a continuous probability distribution
function, then, by Lemma 4.1.39 of \cite{MR3152785},
\[
\lim_{n\to\infty}\mathbb{P}\bigg(\frac{\log Z(\kappa n,n)-n\bar{f}_{k,\theta,\kappa}}{n^{1/3}}\le r\bigg)=p_{k,\theta,\kappa}(r) \;.
\]

On account of this, in order to prove Theorem \ref{thm:main} it suffices to show that for $\bar f_{k,\theta,\kappa}$ from (\ref{fktk}) and $\bar g_{k,\kappa}$ from (\ref{gktk}),
\begin{equation}\label{eqlimittotake}
\lim_{n\to\infty}\mathbb{E}e^{-uZ(\kappa n,n)} = F_{GUE}\Big(\big(\frac{\bar{g}_{k,\kappa}}{2}\big)^{-1/3}\, r\Big).
\end{equation}
%In fact, the theorem only asks for this result to hold for $\theta=1$ and $\kappa$ sufficiently large. We will, for the moment, proceed with $\theta$ general and only set it equal to 1 later to simplify the asymptotic analysis.

In order to prove the limit in (\ref{eqlimittotake}) we utilize the Fredholm determinant formula from Theorem \ref{thm:Fredholm-SWpolymer}. Towards this end, define
\[
G(z)=\log\Gamma(z)-\kappa\log\Gamma(k+z)+\big(\bar{f}_{k,\theta,\kappa}-(\kappa-1)\log\theta\big)\, z  \;.
\]
Then we can rewrite \eqref{eq:Ku} as (recall the relation $k=m_1/\theta$)
\begin{equation}
K_{u}(v,v')
=\frac{1}{2\pi\I}\int_{-\I\infty+1/2}^{\I\infty+1/2}
\frac{\pi}{\sin(\pi(v-\tilde{z}))}
\frac{
	e^{nG(v)+rn^{1/3}v}}
{e^{nG(\tilde{z})+rn^{1/3}\tilde{z}}}
\frac{ \theta^{(\tilde{z}-v)} }{\tilde{z}-v'}
d\tilde{z} \;.
\label{eq:KuG}
\end{equation}

The derivatives of $G$ is given by
\begin{align*}
G'(z)&=\Psi(z)-\kappa\Psi(k+z)+\bar{f}_{k,\theta,\kappa}-(\kappa-1)\log\theta \;,\\
G''(z)&=\Psi'(z)-\kappa\Psi'(k+z) \;.
\end{align*}

\begin{lem} \label{lem:exists-cri}
Given the parameters $k>0$, for every $\bar{z}>0$, provided
that $\kappa$ is sufficiently large, there exists $\bar{t}\in(0,\bar{z})$
such that $\Psi'(\bar{t})-\kappa\Psi'(k+\bar{t})=0$.
\end{lem}

The value $\bar{t}$ depends on $k,\kappa$. We don't write
this dependence explicitly for simplicity of notation.
\begin{proof}
The function $F(z)=\Psi'(z)-\kappa\Psi'(k+z)$ is continuous on $z\in(0,\infty)$.
If $z\to0$, then $F(z)\to\infty$. On the other hand, as $z\to\min(1/4,\bar{z})$,
the quantity $\kappa\Psi'(k+z)$ is bounded below by $\kappa$ times
a constant (depending on $k$), so as long as $\kappa$ is sufficiently
large, and hence $F(z)$ is negative. Therefore there exists $\bar{t}\in(0,\bar{z})$
such that $F(\bar{t})=0$.
\end{proof}
Given a sufficiently large $\kappa$,
let $\bar t=\bar t(k,\kappa)$ be such that $G''(\bar{t})=0$.
Lemma \ref{lem:exists-cri} guarantees that $\bar t$ is small if we assume $\kappa$ large.
One can then choose
\begin{equation}\label{fktk}
\bar{f}_{k,\theta,\kappa}=-\Psi(\bar{t})+\kappa\Psi(k+\bar{t})+(\kappa-1)\log\theta
\end{equation}
so as to make $G'(\bar{t})=0$ as well. Let
\begin{equation}\label{gktk}
\bar{g}_{k,\kappa}=-G'''(\bar{t})
\end{equation}
then formally,
\[
G(z)\approx G(\bar{t})-\frac{\bar{g}_{k,\kappa}}{6}(z-\bar{t})^{3}+h.o.t.,
\]
where $h.o.t.$ stands for higher order terms.
Substitute this into (\ref{eq:KuG}), and make changes of variables
\[
v_1=\big(\tfrac{\bar{g}_{k,\kappa}}{2}\big)^{1/3}n^{1/3}(v-\bar{t}) \;,
\qquad
v_1'=\big(\tfrac{\bar{g}_{k,\kappa}}{2}\big)^{1/3}n^{1/3}(v'-\bar{t}) \;,
\qquad
z_1=\big(\tfrac{\bar{g}_{k,\kappa}}{2}\big)^{1/3}n^{1/3}(\tilde{z}-\bar{t}) \;,
\]
then formally (and for the moment neglecting a discussion of contours), one has (note that $ \theta^{\tilde{z}-v}  \to 1$ as $n\to\infty$)
\[
\lim_{n\to\infty}\mathbb{P}\bigg(\frac{\log Z(\kappa n,n)-n\,\bar{f}_{k,\theta,\kappa}}{n^{1/3}}\le r\bigg)=\det\Big(I+K_r\Big)
\]
where
\[
K_r(v_1,v'_1)=\frac{1}{2\pi i}\int
	\frac{1}{v_1-z_1}
	\frac{\exp\big\{-v_1^{3}/3+(\frac{\bar{g}_{k,\kappa}}{2})^{-1/3}r v_1\big\}}
	{\exp\big\{-z_1^{3}/3+(\frac{\bar{g}_{k,\kappa}}{2})^{-1/3}r z_1\big\}}
	\frac{dz_1}{z_1-v_1'} \;.
\]
The last Fredholm determinant formula (on suitable contours as defined below) is a well-known formula for Tracy-Widom distribution $F_{GUE}\Big(\big(\tfrac{\bar{g}_{k,\kappa}}{2}\big)^{-1/3}\, r\Big)$.

These discussions have formally demonstrated Theorem~\ref{thm:main}.
In the following, we make this derivation rigorous. Note that in the above formal discussions we did not specify the contours. We start with precise definitions of the contours.

\begin{defn}
Define a contour $\mathcal C_{v}$ leaving $\bar{t}$ at an angle $2\pi/3$
as a straight line segment from $\bar{t}$ to $\sqrt{3}\bar{t}i$,
followed by a counter-clockwise circular arc (centered at the origin)
until $-\sqrt{3}\bar{t}i$, then a straight line segment
back to $\bar{t}$.
The contour $\mathcal C_{v}$ is oriented counter-clockwisely.
We also define a contour $\mathcal C_{\tilde{z}}$ which
consists of two rays, symmetric with each other by the real axis,
from $\bar{t}+n^{-1/3}$ leaving at an angle $\pm\pi/3$.
The contour $\mathcal C_{\tilde{z}}$ is oriented so as to have increasing imaginary part.
See Figure
\ref{fig:steep-contours}.
\end{defn}

We shift the contours for $v,v'$ to the contour $\mathcal C_v$
and shift the contour for $\tilde{z}$ as to the contour $\mathcal C_{\tilde z}$.
Provided that $\kappa$ is sufficiently large so that $\bar{t}$ is
sufficiently small, these shifts do not cross the poles.
More precisely, the integrand of the kernel $K_u(v,v')$ contains in its denominator
the factors $\sin(\pi(v-\tilde z))$ which vanishes if $v-\tilde z\in \mathbb{Z}$,
and $\Gamma(\tilde z)$ which is zero at $\tilde z=0,-1,-2,...$,
and $\Gamma(k+v)$ which vanishes at $v=-k.-k-1,-k-2,...$,
and finally the factor $\tilde z-v'$. So
as long as $\bar t$ is small so that the contour $\mathcal C_v$
is sufficiently small,
and $\mathcal C_v $ does not intersect with $\mathcal C_{\tilde z}$,
these points are all avoided.

\begin{center}
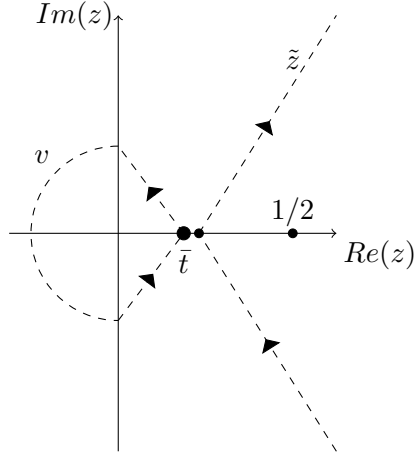
\begin{figure}
\begin{tikzpicture} [scale=2.9]
\begin{scope}[every node/.style={sloped,allow upside down}]
\node at (-0.6,0) {};
\draw [->] (-0.5,0) -- (1,0);
\draw [->] (0,-1) -- (0,1);
\filldraw  (0.3, 0) circle (0.03cm) ;
\node at (0.3,-0.13) {$\bar t$};
\node at (1.2, -0.1) {$Re(z)$};
\node at (-0.2, 1) {$Im(z)$};
%\draw [dashed] (0.5,-1) -- (0.44,-0.2)--(0.42,-0.1) --(0.37,0)--(0.42,0.1)--(0.44,0.2)--(0.5,1);
%\draw [dashed,->] (0.37,0) -- (0.685,0.5); \draw [dashed] (0.685,0.5) -- (1,1);
\draw [dashed] (0.37,0) -- node{\midarrow}  (1,1) ;
\draw [dashed] (1,-1) -- node{\midarrow}  (0.37,0) ;
\filldraw (0.37, 0) circle (0.02cm) ;
\draw [dashed] (0.3,0) -- node{\midarrow}  (0,0.4);
\draw [dashed] (0,-0.4) -- node{\midarrow}  (0.3,0);
\draw [dashed]  (0,0.4) arc (90:270:0.4) ;
%\draw [dashed] (0,0) circle (0.08cm);
\filldraw (0.8, 0) circle (0.02cm) ;
\node at (0.8, 0.1) {$1/2$};
\node at (0.8,0.8) {$\tilde z$};
\node at (-0.35,0.35) {$v$};
%\draw [red, very thick, dashed] (0.3,0) circle (0.13cm) ;
%\node [red] at (0.28,0.3) {$N^{-1/3}$};
\end{scope}
\end{tikzpicture}
\caption{Illustrations of the contours $C_v$ and $C_{\tilde z}$.}
\label{fig:steep-contours}
\end{figure}
\par\end{center}

We will follow the idea from \cite{MR3207195} to parametrize all
other parameters ($\kappa$, $ \bar{f}_{k,\theta,\kappa}$
and $\bar{g}_{k,\kappa}$) by the value of the critical point $\bar{t}$,
and therefore write them as $\kappa_{\bar t}$, $\bar f_{\bar t}$ etc.
%We will take $\theta=1$ for simplicity below, so the parameters $\bar{f}_{k,\theta,\kappa}$
%and $\bar{g}_{k,\theta,\kappa}$ reduce to $\bar{f}_{k,\kappa}$ and
%$\bar{g}_{k,\kappa}$ of Definition~\ref{def:parameters}.

\begin{lem}
\label{lem:asymCv}
Suppose that $\kappa$ is sufficiently large. There
exists constants $c>0,\tilde c>0$ only depending on $\kappa$
such that for all $v\in\mathcal{C}_{v}$
satisfying $|v-\bar{t}|<c$,
\be \label{eq:ultra-close-tbar}
Re(G(v)-G(\bar{t})) \le Re(-\tilde c \, \bar{g}_{k,\kappa}(v-\bar{t})^{3})
\ee
Furthermore, along the part of $\mathcal{C}_{v}$ with $|v-\bar{t}|\geq c$,
one has $Re(G(v)-G(\bar{t}))<c'$ for a strictly negative constant
$c'$.
\end{lem}

\begin{proof}
By Taylor's theorem  and the fact that $G'(\bar t)=G''(\bar t)=0$ and $G'''(\bar t)=-\bar{g}_{k, \kappa}$, one has
\be \label{eq:Taylor-ultraclose}
G(v)=G(\bar{t})-\frac{1}{6}\bar{g}_{k, \kappa} (v-\bar t)^3+o(|v-\bar t|^3) \;,
\ee
so there exists $c>0,\tilde c>0$
only depending on $\bar t$ (or equivalently on $\kappa$)
such that
within the small neighborhood $|v-\bar{t}|<c$,
the bound
\eqref{eq:ultra-close-tbar} holds.

To show the lemma for $v$ on the other part of $\mathcal C_v$,
let $\gamma_{{\rm E}}=-\Psi(1)=0.577...$ be the Euler-Mascheroni
constant. For $v$ small, on has (see \cite[Section~2]{MR3116323})
\be \label{eq:gamma-exp}
\begin{aligned}
\log\Gamma(v)&=-\log v-\gamma_{{\rm E}}v+O(v^{2}) \;,\\
\Psi(v)&=-\frac{1}{v}-\gamma_{{\rm E}}+O(v) \;,\\
\Psi'(v)&=\frac{1}{v^{2}}+O(1) \;.
\end{aligned}
\ee
Utilizing the expansions \eqref{eq:gamma-exp}, and by the choices of $\bar t$ and $\bar f_{\bar t}=\bar f_{k,\theta,\kappa}$ above,
\begin{equation} \label{eq:kappa_cri}
\kappa_{\bar{t}}=\frac{\Psi'(\bar{t})}{\Psi'(k+\bar{t})}
=\Psi'(k+\bar{t})^{-1}\big(\bar{t}^{-2}+O(1)\big) \;,
\end{equation}
\begin{equation}  \label{eq:f_cri-theta}
\bar f_{\bar{t}} - (\kappa-1)\log\theta
=\kappa_{\bar{t}} \, \Psi(k+\bar{t})-\Psi(\bar{t})
=\frac{\Psi(k+\bar{t})}{\Psi'(k+\bar{t})}\big(\bar{t}^{-2}+O(1)\big)
	+\bar{t}^{-1}+\gamma_{{\rm E}} +O(\bar t) \;.
\end{equation}

We then first show that for $\bar t$ small enough, there exists a constant $c_1 \in(0,1)$ independent of $\bar t$, such that if $c\le |v-\bar t|<c_1\bar t$ then
$Re(G(v)-G(\bar{t}))$ is bounded by a strictly negative constant.
In fact, one has
\[
G'''(v)=\Psi''(v)-\kappa_{\bar t}\,\Psi''(k+v)=\Psi''(v)-\Psi'(\bar t)\Psi'(k+\bar t)^{-1}\Psi''(k+v) \;,
\]
and if $\bar t \to 0$ then
\[
\Psi'(k+\bar t)^{-1}\Psi''(k+v) \to \Psi'(k)^{-1}\Psi''(k)
\]
which is a negative constant, call it $-b_1$; and for $\bar t$ small enough, one has $Re(\Psi''(v))<0$  and its absolute value is much larger than
$b_1\Psi'(\bar t)$.
So one can choose a universal constant $c_1$ (for instance $c_1=1/2$) such that the function $Re(G'''(v))<-b_2$ if $ |v-\bar t|<c_1\bar t$  for some $b_2>0$.
Therefore by the integral form of Taylor's remainder theorem
and $G'(\bar t)=G''(\bar t)=0$,
\[
Re(G(v)-G(\bar t))=\frac{(v-\bar t)^3}{2}\int_0^1 (1-s)^2 Re(G'''(\bar t+s(v-\bar t))) ds
<\frac{(v-\bar t)^3}{6} (-b_2)\;.
\]
Noting that $(v-\bar t)^3>0$ one obtains the claimed bound.

For the region $|v-\bar t|\ge c_1\bar t$,
expanding
\[
\log\Gamma(k+v) = \log\Gamma(k) +v\Psi(k) +\frac{1}{2} v^2\Psi'(k)+O(v^3)
\]
and expanding
$\log\Gamma(k+\bar t)$ around
$\bar t=0$ similarly, one has by the definition of $G$
\begin{equation}\label{eq:G-G}
\begin{aligned} 
& G(v)  - G(\bar{t})
=\Big(\log\Gamma(v)-\kappa_{\bar{t}} \, \log\Gamma(k+v) 
	+ \Big(\bar f_{\bar{t}} - (\kappa-1)\log\theta \Big) \, v\Big)\\
 &\qquad\qquad\qquad
 -\Big(\log\Gamma(\bar{t})-\kappa_{\bar{t}} \, \log\Gamma(k+\bar{t})
 	+ \Big( \bar f_{\bar{t}} -(\kappa-1)\log\theta \Big) \, \bar{t}\Big)\\
 & =-(\log v-\log\bar{t})-\gamma_{{\rm E}}(v-\bar{t})+O(\bar{t}^{2})\\
 & \qquad
 -\kappa_{\bar{t}} \, \Big(
 	(v-\bar{t})\Psi(k) +\frac{1}{2}(v^2-\bar{t}^2)\Psi'(k)
	 +O(\bar{t}^3)
 \Big) + \Big(\bar f_{\bar{t}}-(\kappa-1)\log\theta \Big) \, (v-\bar{t})
\end{aligned}
\end{equation}
where we all the error terms $O(v^n)$ have been replaced by $O(\bar t^n)$ since every point  $v\in \mathcal C_v$ is of order $\bar t$.

By the expressions for $\kappa_{\bar{t}}$ and $\bar f_{\bar t}$ in (\ref{eq:kappa_cri}) and (\ref{eq:f_cri-theta}), one has
\[
\begin{aligned}
-\kappa_{\bar{t}} \, \Psi(k)  + \bar f_{\bar{t}} -(\kappa-1)\log\theta
  &= \frac{\Psi(k+\bar t)-\Psi(k)}{\Psi'(k+\bar t)} \big(\bar t^{-2}+O(1)\big) + \bar t^{-1}+\gamma_{\rm E}+O(\bar t) \\
&=(\bar t+O(\bar t^2)) (\bar t^{-2}+O(1))+\bar t^{-1} +O(1) \\
&=2\bar{t}^{-1}+O(1)
\end{aligned}
\]
and
\[
\kappa_{\bar t} \, (v^2-\bar t^2)\Psi'(k)=c_2 \bar t^{-2} (v+\bar t)(v-\bar t)+O(\bar t^2)
\]
with $|c_2-1|$ arbitrarily small as $\bar t$ sufficiently small.
Substituting the above two identities into \eqref{eq:G-G},
and noting that there exists $c_3>0$ such that
$2\bar t^{-1}+c_2\bar t^{-2}(v+\bar t)<c_3\bar t^{-1}$, we find
\begin{equation}
  Re\big(G(v)-G(\bar{t})\big)< c_3 \bar{t}^{-1}(v-\bar{t}) - \log\frac{v}{\bar t} +O(\bar t)
\label{eq:Gz-Gc}
\end{equation}
Since along the part of $\mathcal{C}_{v}$
with $|v-\bar t|\ge c_1\bar t$,
 the $O(\bar{t})$ error is dominated by
the other two terms which are both $O(1)$,
and $|v|>|\bar t|$ so $\log(v/\bar t)>0$,
therefore $Re(G(v)-G(\bar{t}))$
is bounded by a strictly negative constant.
\end{proof}

\begin{lem}
\label{lem:asym-Cz}
Suppose that $\kappa$ is sufficiently large.
There exists constants $c>0,\tilde c>0$ such that for all $z\in\mathcal{C}_{\tilde{z}}$
satisfying $|z-\bar{t}|<c$,
\be \label{eq:inf-cont-local}
Re\big(G(z)-G(\bar{t})\big)\geq Re(-\tilde c \, \bar{g}_{k, \kappa}(z-\bar{t})^{3})
\ee
Along the part of $\mathcal{C}_{\tilde{z}}$ such that $|z-\bar{t}|\geq c$,
one has
\be \label{eq:inf-contour}
Re\big(G(z)-G(\bar{t})\big)\geq c'Re(z)
\ee
for some constant $c'>0$.
\end{lem}
\begin{proof}
For $z\in\mathcal{C}_{\tilde{z}}$ in a sufficiently small neighborhood
of $\bar{t}$, namely $|z-\bar t|<c$ for a constant $c>0$, the bound \eqref{eq:inf-cont-local} follows from the Taylor's theorem
in the same way as \eqref{eq:Taylor-ultraclose} in the proof of  Lemma~\ref{lem:asymCv}; note that
$(z-\bar t)^3$ is now negative for $z\in\mathcal C_{\tilde z}$.

For $z$
outside this small neighborhood but within $O(\bar{t}^{1/2})$ distance
from $\bar{t}$, the argument is the same as
in the case of $|z-\bar t|<O(\bar t)$ in the proof of
Lemma \ref{lem:asymCv},
namely we can  show that in this region $G'''(z)<0$ and note that now
$(z-\bar t)^3<0$  for $z\in\mathcal C_{\tilde z}$.
%\noteH{The difference between ultra close range and this intermediate range:
%for ultra close range, one only needs to know $G'''(\bar t)<0$,
%but for intermediate range, just like proof of  Lemma~\ref{lem:asymCv},
%one needs to know $G'''(z)<0$ not only at $\bar t$.
%Why do I distinguish this intermediate range and the third range below:
%the proof below has a case "and finally for $j=0$....." where
%I want to show that if $j=0$ the summand on the right of \eqref{eq:der-G} is positive.
%There I have to assume $|z-\bar t|>O(\bar t^{1/2})$. Without this intermediate range
%we would only have $|z-\bar t|>c$ but note that $c$ depends on $\bar t$ in
%an unknown way. The term corresponding to $j=0$ is $y$ times
%$A:=-\frac{1}{x^{2}+y^{2}}+\frac{\kappa}{(x+k)^{2}+y^{2}}$
%and since $\kappa=O(\bar t^{-2})$ (see \eqref{eq:kappa_cri}), if $z$ is too close to $\bar t$ the quantity $A$ is roughly
%$-\bar t^{-2} + O(\bar t^{-2})/k^2 $ which might be very negative.
%}

For all $z$ outside this $O(\bar t^{1/2})$ the proof is as
follows. From \cite[Section~5.2]{MR3207195}
\[
Re\big(\log\Gamma(x+iy)\big)=\sum_{j=0}^{\infty}\left(\frac{x}{j+1}-\frac{1}{2}\ln\left((x+j)^{2}+y^{2}\right)+\ln(j)\mathbf{1}_{j\geq1}\right)-\gamma_{{\rm E}}x.
\]
Write $z=x+iy$. Assume that $y>0$, and we will show that the derivative with respect to $y$ of the left side of  \eqref{eq:inf-contour} is bounded below by a positive number, which immediately yields \eqref{eq:inf-contour}. The proof for $y<0$ follows in the same way since $G(x+iy)$ is even in $y$.
Taking derivative, one has
\be \label{eq:der-G}
\frac{\partial}{\partial y}Re\big(G(z)-G(\bar{t})\big)
=\sum_{j=0}^{\infty}\bigg(-\frac{y}{(x+j)^{2}+y^{2}}+\frac{\kappa y}{(x+k+j)^{2}+y^{2}}\bigg) \;.
\ee
Define a constant $C(k,x,y,j)=\big((x+k+j)^2+y^2 \big) / \big((x+j)^2+y^2\big)$.
There exists a constant $C'$ (depending on the fixed shape parameter $k$ of the Gamma distribution) such that:
\begin{itemize}
\item if $\max(x,|y|,j)>C'$ then $C(k,x,y,j)<2$;
\item within the compact domain
\[
\{(x,y,j): j \geq 1 \mbox{ and } \max(x,|y|,j)\le C' \}  \;,
\]
the continuous  function $C(k,x,y,j)$ (regarding $j$ as a real number) is bounded by a
constant,
and within the compact domain
\[
 \{(x,y,j): j=0 \mbox{ and } |z-\bar t|\geq 1  \mbox{ and } \max(x,|y|)\le C'\}
 \]
since $(x+j)^2+y^2$ is bounded away from zero, the continuous  function $C(k,x,y,0)$ is again bounded by a
constant  independent of $x,y,j$;
\item
and finally for $j=0$ and $O(\bar t^{1/2})<|z-\bar t| \le 1$,
since we have shown in \eqref{eq:kappa_cri} that
$\kappa=O(\bar t^{-2})$, the second term on the right of \eqref{eq:der-G}
is $O(\bar t^{-3/2})$ which dominates over the first term there which is $O(\bar t^{-1/2})$.
\end{itemize}
Therefore if $\kappa$ is sufficiently large,
every summand on the right of \eqref{eq:der-G} is positive.

We show that summing over sufficiently many (positive) terms
on the right of \eqref{eq:der-G} will give a quantity bounded below
by a positive constant independent of $x,y$.
In fact, within a compact domain the right side of \eqref{eq:der-G}
is bounded below
by a positive constant. And outside this compact domain,
the right side of \eqref{eq:der-G} is bounded below
by $(\frac{\kappa}{2}-1)y \sum_{j} \frac{1}{(x+j)^{2}+y^{2}}$.
The sum over $j$ from $0$ to $J$
is estimated by $\frac{1}{y} \Big(\arctan\frac{J+x}{y}-\arctan\frac{x}{y}\Big)$.
The factor $\frac{1}{y} $ cancels with the factor $y$ outside the sum and we obtain a strictly positive number. Therefore the desired bound
holds.
\end{proof}

\begin{proof}[Proof of Theorem~\ref{thm:main}]
Given Lemma~\ref{lem:asymCv} and Lemma~\ref{lem:asym-Cz}
the proof of the theorem is standard (see for instance \cite{MR2796514}, \cite{MR3207195} or \cite{MR3116323}). 
Indeed the above two lemmas show that for any $\eps>0$,
one can restrict to finite number of leading terms in the Fredholm series expansion, and localize the integrals in these leading terms to a window of size $n^{-1/3}$ around the critical point, both approximations causing errors that are bounded by $\eps/3$ uniformly in $n$. Rescaling the window by $n^{1/3}$, the integrals from $K_u$ then converge to the integrals from $K_r$, which is essentially shown in the beginning of the section.
\end{proof}

\section{Replica method for strict-weak polymer} \label{sec:Replica}

Given $\overrightarrow{n}=(n_{1}\ge...\ge n_{k})$, consider
the moments
\begin{equation}
u(t,\overrightarrow{n})=\mathbb{E}\bigg[\prod_{i=1}^{k}Z(t,n_{i})\bigg] \;.
\label{eq:def-u}
\end{equation}
In this section we will show an explicit formula for $u(t,\overrightarrow{n})$,
which is stated in Theorem~\ref{thm:mom-formu} below.
Define an operator
\[
\tau^{(i)}u(t,\overrightarrow{n})=\tau^{(i)}u(t,n_{1},...,n_{k})
=u(t,n_{1},...,n_{i}-1,...,n_{k})\;.
\]

\begin{lem}\label{lemtrueevol}
$u(t,\overrightarrow{n})$ solves the following evolution equation
\begin{equation}
u(t+1,\overrightarrow{n})=\sum_{A\subset\{1,...,k\}}\Big(\prod_{i=1}^{\ell}m_{c_{i}^{A}}\Big)\,\Big(\prod_{i=1}^{\ell}
\prod_{j \,=\,  c_1+...+c_{i-1} \atop +\, c_i^A+1}^{ c_1+...+c_i} \tau^{(j)}\Big)
\, u(t,\overrightarrow{n})
\label{eq:true-ev}
\end{equation}
where $\overrightarrow{n}$ is such that
\[
n_{1}=...=n_{c_{1}} \; > \; n_{c_{1}+1}=...=n_{c_{1}+c_{2}} \; > \;
\cdots\cdots
\; > \; n_{c_{1}+...+c_{\ell-1}+1}=...=n_{c_{1}+...+c_{\ell}}
\]
for some positive integers $c_{1},...,c_{\ell}$ so that $\sum_{i=1}^{\ell}c_i=k$,
and $ $
\[
c_{i}^{A}=\#\,\Big( \Big\{ \sum_{j=1}^{i-1} c_j+1,\; \sum_{j=1}^{i-1} c_j+2,\;
... \;,\; \sum_{j=1}^{i-1} c_j+c_{i} \Big\}\,\cap\, A\Big)
\]
where $\#$ means the number of elements in a set, and finally $m_{i}$
is the $i$-th moment of a Gamma random variable with shape parameter
$k$ and scale parameter $\theta$ (with the convention $m_0=1$).
%\[
%u(t+1,\overrightarrow{n})=\prod_{i=1}^{k}(m_{1}+\tau^{(i)})\, u(t,\overrightarrow{n})
%\]
\end{lem}
\begin{proof}
By the recursive relation (\ref{eq:Recursive-Z}),
\begin{equation}
\begin{aligned}\prod_{i=1}^{k} & Z(t+1,n_{i})=\prod_{i=1}^{k}\Big(Y(t,n_{i})\, Z(t,n_{i})+Z(t,n_{i}-1)\Big)\\
 & =\sum_{A\subset\{1,...,k\}}\prod_{i\in A}\Big(Y(t,n_{i})Z(t,n_{i})\Big)\prod_{i\notin A}Z(t,n_{i}-1) \;.
\end{aligned}
\label{eq:prodZ}
\end{equation}
%For any subset $A\subset\{1,...,k\}$, define.
Taking expectations
on both sides, and noting that $Z(t,-)$ is independent of $Y(t,-)$
and $Y(t,n)$ are i.i.d. for different $n$, one has
\[
u(t+1,\overrightarrow{n})=\sum_{A\subset\{1,...,k\}}\Big(\prod_{i=1}^{\ell}m_{c_{i}^{A}}\Big)\,\mathbb{E}\Big[\prod_{i\in A}Z(t,n_{i})\prod_{i\notin A}Z(t,n_{i}-1)\Big] \;.
\]
Note that the last expectation can be written in terms of $u$ by
rearranging the $n$ variables into non-increasing order (see the
definition (\ref{eq:def-u}) of $u$). Within each ``cluster''
consisting of $c_{i}$ identical variables
\[
n_{c_{1}+...+c_{i-1}+1}=...=n_{c_{1}+...+c_{i-1}+c_{i}} \;,
\]
the variables $n_{i}$ with $i\notin A$ are subtracted by $1$ so
they must be rearranged to the right of the other $n_{i}$ with $i\in A$,
resulting in
\[
\begin{aligned}
u(t,\; \cdots\cdots,\;  &   n_{c_{1}+...+c_{i-1}+1},...,n_{c_{1}+...
	+ c_{i-1}+c_{i}^{A}}\;,\\
 & n_{c_{1}+...+c_{i-1}+c_{i}^{A}+1}-1,...,n_{c_{1}+...+c_{i}}-1 \;,\;\cdots\cdots) \;.
\end{aligned}
\]
Therefore we obtain (\ref{eq:true-ev}).
\end{proof}
By the Laplace transform of the Gamma distribution given in (\ref{eq:Laplace-Gamma}),
one has $m_{i}=\theta^{i}\prod_{j=0}^{i-1}(k+j)$.

Let us momentarily consider the true evolution equation in Lemma \ref{lemtrueevol} when $k=1$ and $2$. For $k=1$ we necessarily have $\ell=1$ and $c_1=1$. If $A=\emptyset$,  then $m_{c_1^A}=m_0=1$,  and the product $\prod_j \tau^{(j)}$
is simply $\tau^{(1)}$. If $A=\{1\}$, then $m_{c_1^A}=m_1$,  and no factor contributes to the product $\prod_j \tau^{(j)}$. Therefore
\[
u(t+1,n)=\tau^{(1)}u(t,n)+m_1 u(t,n) \;.
\]
For $k=2$ and when $n_1<n_2$,  we have $\ell=2$ and $c_1=c_2=1$. It is straightforward to check that the cases $A=\emptyset$, $\{1\}$, $\{2\}$ and $\{1,2\}$ give the four terms on the right side below
\[
u(t+1,n_1,n_2)=\Big( \tau^{(1)}\tau^{(2)}+m_1 \tau^{(2)}
+ m_1 \tau^{(1)} + m_1^2 \Big) u(t,n_1,n_2) \;.
\]
And for $k=2$ and when $n_1=n_2$, we have $\ell=1$ and $c_1=2$. One can check that the cases $A=\emptyset$, $\{1\}$, $\{2\}$ and $\{1,2\}$ give the four terms on the right side below
\[
u(t+1,n_1,n_2)=\Big( \tau^{(1)}\tau^{(2)}+m_1 \tau^{(2)}
+ m_1 \tau^{(2)} + m_2 \Big) u(t,n_1,n_2) \;.
\]

Note that for general $k$ and $n_{1}>...>n_{k}$, one has
\begin{equation}
u(t+1,\overrightarrow{n})=\sum_{A\subset\{1,...,k\}}m_{1}^{|A|}\Big(\prod_{i\notin A}\tau^{(i)}\Big)u(t,\overrightarrow{n})=\prod_{i=1}^{k}(m_{1}+\tau^{(i)})\, u(t,\overrightarrow{n})\label{eq:free-ev}
\end{equation}
which can be derived either from (\ref{eq:true-ev}) or by taking
expectation on (\ref{eq:prodZ}). We call (\ref{eq:free-ev}) the
{\it free evolution equation}, and (\ref{eq:true-ev}) the {\it true evolution
equation}. Using the below reduction of the true evolution to the free evolution, it is possible to diagonalize the true evolution equation via coordinate Bethe ansatz. We do not pursue this further here, but reference, for example \cite{BCPS1,BCPS2,TTPLD}.

\begin{lem}
Suppose that $u(t,\overrightarrow{n})$ solves the free evolution
equation (\ref{eq:free-ev}) for all $\overrightarrow{n}=(n_{1}\ge...\ge n_{k})$,
and satisfies the following two-body boundary conditions
\[
\Big(m_{1}^{2}-m_{2}
	+m_{1}(\tau^{(i)}-\tau^{(i+1)})\Big)\, u(t,\overrightarrow{n})=0
	\qquad(n_{i}=n_{i+1}) \;.
\]
Then $u(t,\overrightarrow{n})$ solves the true evolution
equation (\ref{eq:true-ev}).
\end{lem}

\begin{proof}
Since $m_{1}=\theta k$, $m_{2}=\theta^{2}k(k+1)$, the two-body boundary
conditions can be re-written as
\[
\big(\tau^{(i)}-\tau^{(i+1)}\big) u(t,\overrightarrow{n}) =\theta u(t,\overrightarrow{n})
	\qquad(n_{i}=n_{i+1}) \;.
\]
It suffices to show that for a ``cluster''
\[
n_{1}=...=n_{c}> n_{c+1}\ge...
\]
one has
\[
\prod_{i=1}^{c}(m_{1}+\tau^{(i)}) u(t,\overrightarrow{n})
= \sum_{A\subset\{1,...,c\}}m_{|A|}
	\prod_{j=|A|+1}^{c}\tau^{(j)} u(t,\overrightarrow{n})\;.
\]
Apply the moments formula of Gamma random variables and the boundary
conditions. The above equation can be written as
\[
\prod_{i=1}^{c} \Big( \theta k+\tau^{(c)}+(c-i)\theta \Big) u(t,\overrightarrow{n})
=\sum_{A\subset\{1,...,c\}}\theta^{|A|}\prod_{i=1}^{|A|}(k+i-1)
	\prod_{j=|A|+1}^{c} (\tau^{(c)}+(c-j)\theta) u(t,\overrightarrow{n}) \;.
\]
Observe that each summand on the right hand side only depends on $A$
via $|A|$. So the above identity is equivalent to
\[
\prod_{i=1}^{c} \Big(\theta k+\tau^{(c)}+(c-i) \theta\Big)
=\sum_{a=0}^{c} {c \choose a} \prod_{i=1}^{a}(\theta k+(i-1)\theta)
	\prod_{j=1}^{c-a}(\tau^{(c)}+(j-1)\theta) \;.
\]
This identity can now be proved by induction. For $c=1$, both sides are $\theta k+\tau^{(c)}$.
Suppose that it holds for $c$ and we show that it holds for $c+1$,
namely the right hand side multiplied by $\theta k+\tau^{(c)}+c\theta$
is equal to the right hand side with $c$ replaced by $c+1$.

In fact, writing $\theta k+\tau^{(c)}+c\theta=(\theta k+a\theta)+(\tau^{(c)}+(c-a)\theta)$,
one has
\[
\begin{aligned}
& {c \choose a}
	\prod_{i=1}^{a}(\theta k+(i-1)\theta)
	\prod_{j=1}^{c-a}(\tau^{(c)}+(j-1)\theta)\cdot(\theta k+\tau^{(c)}+c\theta)\\
 & ={c \choose a}
 \prod_{i=1}^{a+1}(\theta k+(i-1)\theta)
 	\prod_{j=1}^{c-a}(\tau^{(c)}+(j-1)\theta)\\
 & \qquad
 +{c \choose a}
 	\prod_{i=1}^{a}(\theta k+(i-1)\theta)
	\prod_{j=1}^{c-a+1}(\tau^{(c)}+(j-1)\theta) \;.
\end{aligned}
\]
Summing over $a$ from $0$ to $c$, and combining pairs of same terms,
one has
\[
\begin{aligned}
& \sum_{a=0}^{c} {c \choose a} \prod_{i=1}^{a+1}(\theta k+(i-1)\theta)
	\prod_{j=1}^{c-a}(\tau^{(c)}+(j-1)\theta)\\
 & \qquad
 + {c \choose a} \prod_{i=1}^{a}(\theta k+(i-1)\theta)
 \prod_{j=1}^{c-a+1}(\tau^{(c)}+(j-1)\theta)\\
 & =\prod_{j=1}^{c+1}(\tau^{(c)}+(j-1)\theta)
 	+\prod_{i=1}^{c+1}(\theta k+(i-1)\theta)\\
 & \qquad
 +\sum_{a=0}^{c-1}\Big( {c \choose a} + {c \choose a+1} \Big)\prod_{i=1}^{a+1}(\theta k+(i-1)\theta)\prod_{j=1}^{c-a}(\tau^{(c)}+(j-1)\theta) \;.
\end{aligned}
\]
Therefore the identity holds by using ${c \choose a}+{c \choose a+1}={c+1 \choose a+1}$.
\end{proof}

\begin{thm} \label{thm:mom-formu}
For $n_1\geq n_2\geq \cdots \geq n_k$, one has the following moment formula
\[
u(t,\overrightarrow{n})
=\frac{1}{(2\pi\I)^{k}}
\int\cdots\int \prod_{1\leq A<B\leq k} \frac{z_{A}-z_{B}}{z_{A}-z_{B}-\theta}
\prod_{j=1}^{k}z_{j}^{1-n_{j}}(m_{1}+z_{j})^{t}
\frac{dz_{j}}{z_{j}}
\]
%with $c=\frac{m_{2}-m_{1}^{2}}{m_{1}}=\theta$
where the contour for $z_k$ is a small circle around the origin,
and the contour for $z_j$ contains the contour for $z_{j+1}+\theta$ for
all $j=1,...,k-1$, as well as the origin.
\end{thm}

The following picture illustrates the choices of contours
in the above integrals. The solid lines are contours for $z_j$ ($j=1,...,k$).
The smallest dashed contour is for $z_k+\theta$,
and is contained in the contour for $z_{k-1}$.
The slightly larger dashed contour is for $z_{k-1}+\theta$,
and is contained in the contour for $z_{k-2}$, etc.
These choices avoid the poles of the integrand.

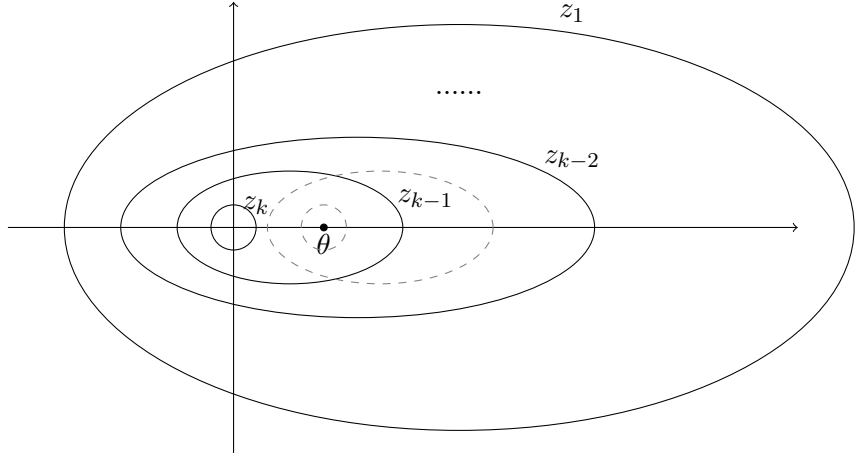
\begin{center}
\begin{tikzpicture} [scale=1.5]
\draw [->] (-2,0)--(5,0);
\draw [->] (0,-2)--(0,2);
\draw (0,0) circle (0.2);
	\node at (0.2,0.2) {$z_k$};
\filldraw  (0.8, 0) circle (0.03cm) ;
	\node at (0.8,-0.15) {$\theta$};
\draw[dashed,gray] (0.8,0) circle (0.2);
\draw (0.5,0) ellipse (1 and 0.5);
	\node at (1.7,0.25) {$z_{k-1}$};
\draw [dashed,gray] (1.3,0) ellipse (1 and 0.5);
\draw (1.1,0) ellipse (2.1 and 0.8);
	\node at (3,0.6) {$z_{k-2}$};
\draw (2,0) ellipse (3.5 and 1.8);
	\node at (3,1.9) {$z_1$};
\node at (2,1.2) {$......$};
\end{tikzpicture}
\end{center}

\begin{proof}
Since
\[
(m_{1}+\tau^{(i)})\, z_{j}^{-n_{j}}=(m_{1}+z_{j})\, z_{j}^{-n_{j}} \;,
\]
we immediately obtain that $u(t,\overrightarrow{n})$ satisfies the free evolution
equation.

To show that the boundary conditions are satisfied, we apply $m_{1}^{2}-m_{2}+m_{1}(\tau^{(i)}-\tau^{(i+1)})$
to the right hand side (with $n_{i}=n_{i+1}$) yields a factor $m_{1}^{2}-m_{2}+m_{1}(z_{i}-z_{i+1})$
which cancels (up to $m_1$) the same factor $z_i - z_{i+1}-\theta=z_i - z_{i+1}-\frac{m_2-m_1^2}{m_1}$ in the denominator. Thus we can deform
the contours for $z_{i}$ and $z_{i+1}$ together, and the factor $z_{i}-z_{i+1}$
in the numerator shows that the integral is zero.

For the initial condition at $t=0$, observe that if $n_{1}>1$ there
is no pole at $z_{1}=\infty$ so the integral is zero; if $n_{k}<1$
there is no pole at $z_{k}=0$ so the integral is again zero. Since
$n_{1}\ge...\ge n_{k}$, for the integral to be nonzero one must have
$n_{1}=...=n_{k}=1$, in which case
\[
u(t,\overrightarrow{n})=
\frac{1}{(2\pi\I)^{k}}
\int\cdots\int \prod_{1\leq A<B\leq k} \frac{z_{A}-z_{B}}{z_{A}-z_{B}-\theta}
\frac{dz_{j}}{z_{j}} \;.
\]
We integrate $z_{k},...,z_{1}$ one by one. Using residue formula
at $z_{k}=0$, $z_{k-1}=\theta$, $z_{k-2}=2\theta$ etc. one obtains $u(t,\overrightarrow{n})=1$ for $t=0$ and $\overrightarrow n=(1,...,1)$.
\end{proof}

These moments may grow too quickly to recover the Laplace
transform of the polymer free energy. This is why we show the convergence
of geometric $q$-TASEP to our polymer model and apply the $e_q$-Laplace transform formula
for geometric $q$-TASEP in the previous sections.

Let us observe a $q$-deformation of the above moment
formula:
\begin{equation} \label{eq:geo-moments}
u^{q}(t,\overrightarrow{n})
=\frac{(-1)^k q^{\frac{k(k-1)}{2}}}{(2\pi \I)^k}
 \int\cdots\int \prod_{1\leq A<B\leq k} \frac{z_{A}-z_{B}}{z_{A}-qz_{B}}
\prod_{j=1}^{k}(1-z_{j})^{-n_{j}}
(1-\alpha z_{j})^{t}
\frac{dz_{j}}{z_{j}} \;.
\end{equation}
The contour of $z_i$ contains the contour of $q \,z_{i+1}$ and $1$.

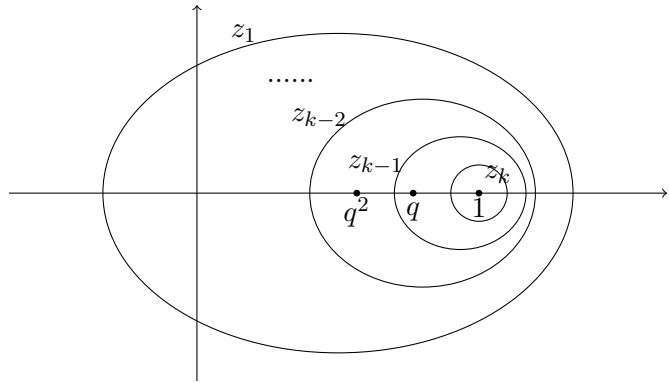
\begin{center}
\begin{tikzpicture}  [scale=1.25]
\draw [->] (-2,0)--(5,0);
\draw [->] (0,-2)--(0,2);
\draw (3,0) circle (0.3);
	\node at (3.2,0.2) {$z_k$};  \filldraw  (3,0) circle (0.03cm) ;
			\node at (3,-0.15) {$1$};
\filldraw  (2.3, 0) circle (0.03cm) ;
	\node at (2.3,-0.2) {$q$};
%\draw[dashed,gray] (0.8,0) circle (0.2);
\draw (2.8,0) ellipse (0.7 and 0.6);
	\node at (1.9,0.3) {$z_{k-1}$};
\filldraw  (1.7, 0) circle (0.03cm) ;
	\node at (1.7,-0.2) {$q^2$};
%\draw [dashed,gray] (1.3,0) ellipse (1 and 0.5);
\draw (2.4,0) ellipse (1.2 and 1);
	\node at (1.3,0.8) {$z_{k-2}$};
\draw (1.5,0) ellipse (2.5 and 1.7);
	\node at (0.5,1.7) {$z_1$};
\node at (1,1.2) {$......$};
\end{tikzpicture}
\end{center}

In fact if we scale
\[
z_{j}=e^{-\eps\tilde{z}_{j}} \;,
\qquad q=e^{-\eps \theta} \;,
\qquad\alpha=e^{-\eps m_{1}} \;,
\]
then as $\eps\to0$,
\[
\frac{z_{A}-z_{B}}{z_{A}-qz_{B}}\to\frac{\tilde{z}_{A}-\tilde{z}_{B}}{\tilde{z}_{A}-\tilde{z}_{B}-\theta} \;,
\qquad
\eps^{n_j} (1-z_{j})^{-n_j} \to \tilde{z}_{j}^{-n_j} \;,
\]
\[
\frac{dz_j}{z_j} = -\eps \tilde z_j \frac{d\tilde z_j}{\tilde z_j} \;,
\qquad
\eps^{-t} (1-\alpha z_{j})^{t}\to (m_{1}+\tilde{z}_{j})^{t} \;,
\]
which shows that
\[
(\prod_{i=1}^{k} \eps^{n_{i}-1-t})\cdot u^{q}(t,\overrightarrow{n})\to u(t,\overrightarrow{n}) \;.
\]
Note that the formula (\ref{eq:geo-moments}) has appeared in \cite{discrete-time}(Theorem
2.1, case (2)), as describing moments for geometric $q$-TASEP (take $a_i=1$ and $\alpha_s=\alpha$ there):
\begin{equation} \label{eq:momthmeq}
\mathbb E \left[ \prod_{i=1}^{k} q^{X_{n_i}(t)+n_i} \right]
= u^{q}(t,\overrightarrow{n}).
\end{equation}
The scalings above are consistent with those of Theorem \ref{thm:rec-rel}.

\section{Stationary polymer} \label{sec:stat}

In this section we introduce a stationary version of the strict-weak polymer model.
Our notation is that  $\N=\{1,2,3,\dotsc\}$ and $\Z_+=\{0,1,2,\dotsc\}$.
For any $x=(t,n)\in\Z\times\Z$, we will sometimes write functions such as $Z(t,n)$
in the form $Z_x$ for simplicity of notation.
If $x$ and $y$ are nearest neighbor points both in $\Z\times\Z$,
then we denote by $(x,y)$ the edge between them.
We denote by $e_1,e_2$ the unit vectors
in the $t$ and $n$ directions, respectively.
%
%\bigskip
%
%Fix $0<\rho, b<\infty$ and let  $\{Y_x\}_{x\in\Z^2}$ be IID Gamma($\rho, b$) variables.  This means that their common distribution is  $\mu(dx)=b^\rho\Gamma(\rho)^{-1} x^{\rho-1}e^{-bx}dx$ on $0<x<\infty$.
%Weights $w_e$ are placed on two types of edges, horizontal $e=(x-e_1, x)$ and diagonal  $e=(x-e_1-e_2, x)$,  according to this definition:
%\be\label{t:w_e}  w_e=\begin{cases}  Y_x, &e=(x-e_1, x)\\ 1, &e=(x-e_1-e_2, x). \end{cases}\ee
%
%Admissible paths $\pi$ on $\Z^2$ use horizontal and diagonal edges.
%Provided $m-a\ge n-b$, when have the partition function
%\be\label{t:Z1} Z_{(a,b),(m,n)}=\sum_{\pi: (a,b)\to(m,n)}  \prod_{e\in\pi} w_e.  \ee
%
%

% The weights initially  given are  $\{U_{i,0}, V_{0,j}, Y_{i,j}: i,j\in\N\}$, all independent,  with $\{Y_{i,j}\}$ as before, and the following distributions on the boundary weights: \be\label{t:UV3}     \text{$U_{i,0}\sim$ Gamma($\alpha+\rho, b$),  \quad  $V_{0,j}^{-1}\sim$ Gamma($\alpha, b$).} \ee

\begin{defn}
The stationary process $\{Z_x^*=Z^*(t,n)\}_{x=(t, n)\in \Z_+\times \N}$ of partition functions
is parametrized by $0<\beta,k,\theta<\infty$ and
 defined as follows.   Admissible paths $\pi^*$ emanate from the point $(0,1)$ and are  allowed three types of edges:  horizontal edges $(x-e_1,x)$ or diagonal edges $(x-e_1-e_2,x)$ as in Definition~\ref{def:SWpoly}, as well as
%(i) horizontal edges  $e=(x-e_1, x)$, (ii)  diagonal edges $e=(x-e_1-e_2, x)$,  and (iii)
  vertical edges along the $y$-axis:  $e=((0,j), (0,j+1))$ for $j\in\N$.
The  weights on these edges are:
\be
\label{t:w_e4}
d^*_e=
%\begin{cases}
%	Y_x, &e = (x-e_1, x)\text{ for $x\in\N^2$}\\
%	1, &e=(x-e_1-e_2, x)\text{ for $x\in\N^2$} \\
%	\tau_e, &\text{if $e$ is on the boundary of  $\Z_+^2$.}
%%U_{i,0}, &\text{ for $e=((i-1,0),(i,0))$ with $i\in\N$}   \\ V_{0,j}, &\text{ for $e=((0,j-1),(0,j))$ with $j\in\N$.}
%\end{cases}
\begin{cases}
	\tau_e, &\text{if $e=\big((0,j), (0,j+1)\big)$  for some $j\in\N$
			or $e=\big((i,1), (i+1,1)\big)$ for some $i\in\Z_+$}  \\
	d_e,	 &\text{otherwise.}
\end{cases}
\ee
where $d_e$ is defined in Definition~\ref{def:SWpoly}, and
$\{\tau_{((i,1),(i+1,1))}\}_{i\in\Z_+}$ and  $\{\tau_{((0,j),(0,j+1))}\}_{j\in\N}$
are given edge weights on the boundary, and independent of $d_e$.
The distributions of these weights are
\be\label{t:t3}     \text{$\tau_{((i-1,1),(i,1))}\sim$ Gamma($\beta+k, \theta$),  \quad  $\tau_{((0,j-1),(0,j))}^{-1}\sim$ Gamma($\beta, \theta$).} \ee
To paraphrase, admissible paths use weights $d_e$ in the bulk ($t\geq 1$ and $n\geq 2$) and  weights $\tau_e$  on the boundary ($t=0$ or $n=1$). %$x$-axis and  the $y$-axis.
 The partition function  $Z^*(t,n)$ is then defined by
%\be\label{t:Zalpha}  Z^*_x =\sum_{\pi^*: (0,0)\to x} \; \prod_{e\in\pi^*} w^*_e\qquad \text{for $x\in\Z_+^2$.}
%\ee
\be \label{t:Zalpha}
Z^*(t,n)=\sum_{\pi^*:(0,1)\to(t,n)}\prod_{e\in\pi^*}d^*_{e}
\ee
\end{defn}

Note that we still have $Z^*(0,1)=1$, but $Z^*(0,n)\neq 0$ for $n>1$.    These partition functions still satisfy the same recursive relation as \eqref{eq:Recursive-Z}
for $t\geq 1$ and $n\geq 2$
\be\label{t:Z10}   Z^*_x  =  Y_x Z^*_{x-e_1} + Z^*_{x-e_1-e_2}
%\qquad \text{for $x\in\N$.}
\ee
where $Y_x := Y_{(x-e_1,x)}$ are i.i.d.\ Gamma$(k,\theta)$ random variables
as in the previous sections.
Superscript $^*$ is used to distinguish this partition function from Definition \ref{def:SWpoly}.
Extend the definition of the  variables  $\tau_e$
%$U_{i,0}$ and $V_{0,j}$
to the ``bulk'' by defining
\be\label{t:U3}    \tau_{(x,y)} =   \frac{Z^*_y}{Z^*_x} \qquad \text{for all directed nearest-neighbor edges $(x,y)$  with  $x,y\in\Z_+\times \N$.}
%U_x=\frac{Z^*_x}{Z^*_{x-e_1}} \quad\text{and}\quad V_x=\frac{Z^*_x}{Z^*_{x-e_2}} \qquad \text{for $x\in\N^2$}
\ee
This is true also on the boundary, by definition of $Z^*_x$.

  The edge weights $\tau_e$ for edges $e$ in the bulk of $\Z_+\times \N$  can also be defined inductively.  Begin with the given initial weights
%$\{\tau_{((i-1,0),(i,0))}, \tau_{((0,j-1),(0,j))},  Y_x \}_{i,j\in\N, \,x\in\N^2}$
\[
\{\tau_{((i-1,1),\,(i,1))}, \tau_{((0,j),\,(0,j+1))},  Y_x \}_{i,j\in\N, \,x\in\Z_+\times\N}
\]
and  apply repeatedly  the formulas
\be\label{t:ind7}   \tau_{(x-e_1,\,x)}= Y_x + \frac1{\tau_{(x-e_1-e_2,\, x-e_1)}},\quad
 \tau_{(x-e_2,\,x)}=\frac{ Y_x  \tau_{(x-e_1-e_2,\, x-e_1)}+1}{\tau_{(x-e_1-e_2,\, x-e_2)}}. \ee
 Using the recursive relation \eqref{t:Z10}
 one shows inductively that equations   \eqref{t:U3} and \eqref{t:ind7} are equivalent.

 \begin{proposition}
The distribution of the process $\big\{\tau_{(x,y)}: x\in\Z_+\times\N, y\in\{x+e_1, x+e_2\} \big\}$ is invariant under lattice shifts.   In particular, the distribution of the process   $\big\{\tau_{(a+x,\,a+y)}: x\in\Z_+\times\N, y\in\{x+e_1, x+e_2\} \big\}$ is the same for all $a\in\Z_+^2$.
 \end{proposition}

 The stationarity is a consequence of the inductive definition \eqref{t:ind7} of the weights and the  next fact about gamma distributions.    In conjunction with \eqref{t:ind7} the next lemma is applied to   $(U,V,Y)=  ( \tau_{(x-e_1-e_2,\, x-e_2)} , \tau_{(x-e_1-e_2,\, x-e_1)}, Y_x)$ and
 $(U',V')= ( \tau_{(x-e_1,\,x)},   \tau_{(x-e_2,\,x)})$.    The statement for $Y'$ is included in the lemma for the sake of completeness but not needed for our present purposes.

 \begin{lem}  Fix $0<\beta, k,\theta<\infty$.   Let $(U,V,Y)$ be independent random variables with distributions
 \be\label{t:t9}
 \text{\rm $U\sim$ Gamma($\beta+k, \theta$),  \ \
 $V^{-1}\sim$ Gamma($\beta, \theta$), \  and  $Y\sim$ Gamma($k, \theta$).} \ee
 Define $(U',V',Y')$ by
\be\label{t:cell1}   U'=Y+\frac1V, \quad V'=\frac{YV+1}U, \quad Y'= \frac{UVY}{YV+1}. \ee
Then the vectors  $(U',V',Y')$ and   $(U,V,Y)$ are equal in distribution.
 \end{lem}

 \begin{proof}   Rewrite the formulas as
 \be\label{t:cell2}   U'=Y+V^{-1}, \quad (V')^{-1}=U\cdot \frac{V^{-1}}{Y+V^{-1}}, \quad Y'= U\cdot\frac{Y}{Y+V^{-1}}. \ee
 The lemma follows from  two basic facts about the beta-gamma algebra.
 First, if $X\sim$ Gamma($\mu,\theta$) and  $Y\sim$ Gamma($\nu,\theta$) are independent, then  $X+Y$ is independent of the pair $( \frac{X}{X+Y}, \frac{Y}{X+Y})$, and
 \[   X+Y\sim \text{Gamma}(\mu+\nu, \theta), \ \ \frac{X}{X+Y}\sim \text{Beta}(\mu,  \nu), \ \  \text{and} \ \   \frac{Y}{X+Y}\sim \text{Beta}(\nu,  \mu).   \]
 Second,  if  $X\sim$ Gamma($\mu+\nu,\theta$) and $Z\sim$  {Beta}($\mu,  \nu$)  are independent,  then $ZX$ and $(1-Z)X$ are independent with distributions    $ZX\sim$ Gamma($\mu,\theta$) and  $(1-Z)X\sim$ Gamma($\nu,\theta$).
 \end{proof}

\section{Free energy law of large numbers} \label{sec:free-energy}
It is a consequence of Theorem~\ref{thm:main}, that for %$\theta=1$ and 
$\kappa$ large enough,
the law of large numbers limit for the free energy of the strict-weak polymer model
is given by $\bar f_{k,\theta,\kappa}$  % $\bar f_{k,\kappa}$ 
of Definition~\ref{def:SWpoly}.
In this section we explain another approach
%assuming $\theta=1$ 
to identify (and with a little more work, prove) the free energy law of large numbers
\be\label{t:g1}  g^*(t,n)=\lim_{N\to\infty} N^{-1}
	\log Z^*(\fl{Nt},\fl{Nn}), \quad 0<t,n<\infty,
\ee
and
 \be\label{t:g2}  g(t,n)=\lim_{N\to\infty} N^{-1}
 	\log Z_{(0,1)}(\fl{Nt},\fl{Nn}), \quad 0<n\le t<\infty.
\ee
where in general the subscript $x$ in $Z_x(t,n)$ stands for the partition function
of polymers emanating from $x$.

Evaluating $g^*$ is immediate from the law of large numbers. By following ratios \eqref{t:U3} from $(0,1)$ to $(\fl{Nt},\fl{Nn})$,
\begin{align*}
N^{-1}\log Z^*\big(\fl{Nt},\fl{Nn}\big) \;
&= \; N^{-1}\sum_{i=1}^{\fl{Nt}} \log \tau_{((i-1,1),(i,1))}
	\;+\;   N^{-1}\sum_{j=1}^{\fl{Nn}} \log \tau_{((\fl{Nt},\,j-1),(\fl{Nt},\,j))}  \\
 &\longrightarrow t\Psi(k+\beta)-n\Psi(\beta) 
 +(t-n) \log\theta \;,
\end{align*}
where we have used the fact that if $X$ is $\mbox{Gamma}(k,\theta)$ distributed then
$\mathbb{E}(\log X)=\Psi(k)+\log(\theta)$. % and the assumption $\theta=1$.
The two sums are sums of i.i.d.\ random variables, though the sums themselves are correlated with each other.

To compute  $g$, the starting point is the decomposition
\be\label{t:decamp}
%\begin{aligned}
Z^*( \fl{Nt},\fl{Nn} ) \;
= \!\!\!\!  \sum_{k=2}^{\fl{Nn}-\fl{Nt}+1}  \!\!\!\!\!\!
	Z^*(k-1,1)  \,Z_{(k,2)}(\fl{Nt},\fl{Nn})
%&\qquad\qquad\qquad
 \;+\;   \sum_{\ell=1}^{\fl{Nn}} Z^*(0,\ell) \,  Z_{(0,\ell)}(\fl{Nt},\fl{Nn}).
%\end{aligned}
\ee
To be specific, the boundary $Z$-values in the decomposition are simply the products
\[   Z^*(k,1)=\prod_{i=1}^k \tau_{((i-1,1), \,(i,1))} \quad\text{and}\quad
Z^*(0,\ell)=\prod_{j=2}^\ell \tau_{((0,j-1), \,(0,j))}. \]

Take $t=n=1$ in which case the first sum on the right vanishes. $N^{-1}\log$ and limit as $N\to\infty$ convert sums into maximums.  Scale the summation index as $\ell=\fl{Ns}$ to arrive at the following equation:
\begin{align*}
\Psi(k+\beta)-\Psi(\beta)&=\lim_{N\to\infty} N^{-1}\log Z^*(N,N) \\
&=\lim_{N\to\infty} N^{-1}\log
	\sum_{\ell=1}^{N}  Z^*(0,\ell) \,  Z_{(0,\ell)}(N,N)  \\
&=\sup_{0\le s\le 1} \{ -s\Psi(\beta) -s\log\theta + g(1,1-s) \}
\end{align*}
The estimation needed for making these steps rigorous is left to the reader.
Let $t=1-s$ and change variables to $y=\Psi(\beta)+\log\theta$ to get
\[  
\Psi(k+\Psi^{-1}(y-\log\theta)) +\log\theta = \sup_{0\le t\le 1}\{ ty + g(1,t)\}. 
\]
Extend the convex  function $f(t)=-g(1,t)$ to $\R$ by setting $f(t)=\infty$ for $t\notin[0,1]$.  Rewrite the equation above as
\[  \Psi(k+\Psi^{-1}(y-\log\theta)) +\log\theta  = \sup_{t\in\R}\{ ty   -f(t)\}. \]
This extended $f$ is convex and lower semicontinuous, and hence by convex duality, for $0\le t\le 1$,
\be\label{t:60}   
g(1,t)=\inf_{y\in\R}\{ -ty+ \Psi(k+\Psi^{-1}(y-\log\theta))\} +\log\theta
= \inf_{\beta>0}\{-t\Psi(\beta)+ \Psi(k+\beta)\} +(1-t)\log\theta.
\ee
Limit \eqref{t:g2} implies homogeneity $g(t,1)=tg(1,t^{-1})$ for $t\ge 1$, and consequently we also have
\be\label{t:60-1}   g(\kappa, 1)=\kappa g(1,\kappa^{-1}) = \inf_{\beta>0}\{-\Psi(\beta)+ \kappa\Psi(k+\beta)\} +(\kappa-1)\log\theta, 
\quad 1\le \kappa <\infty.
\ee
Note that $g(\kappa,1)$ is equal to  %$\bar{f}_{k,\kappa}$ defined in Definition~\ref{def:parameters} 
$\bar{f}_{k,\theta,\kappa}$ defined in Definition~\ref{def:parameters}  %with $\theta=1$
since $\bar t$ is defined to be the critical value of $\beta$
where the infimum is attained.

\bibliographystyle{plain}
\bibliography{ref}

\end{document}